\documentclass[preprint,jmp]{revtex4-1}
\usepackage{graphicx}
\usepackage{color}
\usepackage{amsfonts}
\usepackage{amsthm}
\usepackage{amsmath}
\usepackage{amssymb}
\usepackage{hyperref}
\usepackage{fullpage}

\newcommand{\sceq}{Schr\"{o}dinger equation\ }

\newcommand{\beqn}{\begin{equation}}
\newcommand{\eeqn}{\end{equation}}
\newcommand{\beq}{\begin{equation*}}
\newcommand{\eeq}{\end{equation*}}

\newtheorem{theorem}{Theorem}[section]
\newtheorem{corollary}[theorem]{Corollary}
\newtheorem{lemma}[theorem]{Lemma}
\newtheorem{proposition}[theorem]{Proposition}

\newtheorem{definition}{Definition}
\newtheorem{remark}{Remark}
\newcommand{\eps}{\varepsilon}

\usepackage{color}
\newcommand{\R}{\mathbb{R}}
\newcommand{\C}{\mathbb{C}}

\newcommand{\sg}{\textrm{sign}}
\newcommand{\lam}{\lambda}

\newcommand{\al}{\alpha}

\newcommand{\Det}{\textrm{det}}

\newcommand{\soo}{\mathfrak{so}}
\newcommand{\con}{{\mathcal C}}

\newcommand{\HHH}{{\mathcal H}}

\begin{document}
\title{\LARGE \bf Time minimal trajectories for two-level quantum systems
with two bounded controls}

\author{Ugo Boscain}
\altaffiliation{Center for Applied Mathematics, Ecole Polytechnique, France
}
\email{ugo.boscain@polytechnique.edu}
\thanks{The work of U. Boscain has been supported  by the European Research Council, ERC
StG 2009 ``GeCoMethods", contract number 239748, by the ANR ``GCM", by the program ``Blanc--CSD"
project number NT09-504490, and by the DIGITEO project CONGEO.}

\author{Fredrik Gr\"onberg}
\altaffiliation{School of Electrical Engineering, Royal Institute of Technology (KTH), Sweden}
\email{gronb@kth.se}
\thanks{}

\author{Ruixing Long}
\altaffiliation{Department of Chemistry, Princeton University,
Princeton, New Jersey, USA}
\email{rlong@princeton.edu}
\thanks{The work of R. Long has been supported by U.S. Department of Energy.}

\author{Herschel Rabitz}
\altaffiliation{Department of Chemistry, Princeton University,
Princeton, New Jersey, USA}
\email{hrabitz@princeton.edu}
\thanks{The work of H. Rabitz has been supported by U.S. Department of Energy.}

%\author{ Ugo Boscain, \and Fredrik Gr\"onberg, \and Ruixing Long, \and Herschel Rabitz
%\thanks{U. Boscain is with the Center for Applied Mathematics, Ecole Polytechnique, France
% ({\tt\small ugo.boscain@polytechnique.edu})
%}
%\thanks{F. Gr\"onberg is with the School of Electrical Engineering, Royal Institute of Technology (KTH), Sweden
%({\tt\small gronb@kth.se})
%}
%\thanks{R. Long and H. Rabitz are with the Department of Chemistry, Princeton University, USA
%({\tt\small rlong, hrabitz@princeton.edu})
%}

%\thispagestyle{empty}

\begin{abstract}
In  this paper we consider the minimum time population transfer problem for a two level quantum system driven by {\em two} external fields with bounded amplitude.  The controls are modeled as real functions and we do not  use the Rotating Wave Approximation.
After projection on the Bloch sphere, we tackle the time-optimal control problem with  techniques of optimal synthesis on 2-D manifolds. Based on the Pontryagin Maximum Principle, we characterize a restricted set of candidate optimal trajectories.  Properties on this set, crucial for complete optimal synthesis, are illustrated by numerical simulations. Furthermore, when the two controls have the same bound and this bound is small with respect to the difference of the two energy levels, we get a complete optimal synthesis up to a small neighborhood of the antipodal point of the starting point.
\end{abstract}

\maketitle

\section{Introduction}\label{sec:intro}

In this paper we apply techniques of optimal synthesis on 2-D manifolds to
the population transfer problem for a two-level quantum 
system  (e.g. a spin $1/2$ particle) driven by two external fields.
Two-level systems are the simplest quantum mechanical models interesting 
for applications, see for instance \cite{allen,choen}.
The dynamics is governed by the time dependent Schr\"odinger equation (in
a system of units such that $\hbar=1$):
\begin{eqnarray}
i\frac{d\psi(t)}{dt}=H(t)\psi(t),
\label{se}
\end{eqnarray}
where $\psi(\cdot)=(\psi_1(\cdot),\psi_2(\cdot))^T:[0,T]\mapsto \C^2$ satisfies
$\sum_{j=1}^2|\psi_j(t)|^2=1$, and
\begin{eqnarray}
H(t):=\left(\begin{array}{cc} -E& \Omega(t)+i\Omega_2(t) \\ \Omega_1(t)-i\Omega_2(t)
&E
\end{array}\right),
\label{hg1}
\end{eqnarray}
where $E$ is a real number ($\pm E$ represent the two energy levels of the 
system). The \underline{controls}
$(\Omega_1(\cdot),\Omega_2(\cdot))$,  assumed to be real valued and different
from zero only in a fixed interval, represent external pulsed fields. The  Hamiltonian without
external fields, i.e., the matrix diag$(-E,E)$, is called {\it drift 
term}.

The goal is to steer the system from the first level (i.e. $|\psi_1|^2=1$) to any other target state in minimal-time and with controls of bounded amplitude,
$$
|\Omega_i(t)|\leq M_i,~{i=1,2}~~\mbox{for every $t\in[0,T]$},
$$
where $T$ is the transfer time, $M_1$ and $M_2$ are two positive real constants representing maximum available amplitudes for the control fields. 
The most interesting target state is of course the second level (i.e. $|\psi_2|^2=1$).
 
 \begin{remark}
The two real controls represent two independent fields acting in
two orthogonal directions. They do not come from the use of the Rotating Wave Approximation
close to the Bohr frequency of the system as it often happens in problems
with two controls.
Each field acts \emph{independently} and has its own bound on the amplitude. As a consequence, the use of the interaction picture does not permit to
eliminate the drift term. More precisely, the system would be driftless in the interaction
picture, but with a control set depending explicitly on time (and not anymore of the form
$\vert\Omega_i\vert\leq M_i$ with $M_i$ \emph{constant}).
\end{remark}

The time optimal problem for two level quantum system with one bounded real control was studied in \cite{B-M}. For the same problem with unbounded control, see  \cite{KBG}. The minimum energy problem with one unbounded control was addressed in \cite{daless}. For the minimum energy problem with two unbounded controls see \cite{q2,daless}. Regarding optimal control problems for  two-level dissipative systems, see \cite{S1,BS}. Surprisingly the  time optimal problem with {\it two} bounded real controls for closed two-level systems has not yet been studied. This is a relevant problem in NMR, see \cite{BS,lap} and references therein.

It is standard to eliminate global phase by 
projecting the system on a two dimensional real sphere $S^2$ (called the 
Bloch Sphere) by means of a Hopf map \cite{B-M}. After setting 
$u_i(t)=\Omega_i(t)/M_i$, the controlled \sceq \eqref{se} becomes a two-input control-affine system on the sphere $S^2$:
\begin{equation}\label{cs-p}
\dot{x}=Fx+u_1G_1x+u_2G_2x,\quad \vert u_i(t)\vert\leq 1,
\end{equation}where $x:=(x_1,x_2,x_3)^T\in\R^3$, $\Vert x\Vert^2=1$, and
\begin{eqnarray*}
F&:=&k\cos\alpha\begin{pmatrix}
0&-1&0\\
1&0&0\\
0&0&0
\end{pmatrix},\quad G_1~:=~k\sin\alpha\sin\beta\begin{pmatrix}
0&0&0\\
0&0&-1\\
0&1&0
\end{pmatrix},\\
G_2&:=&k\sin\alpha\cos\beta\begin{pmatrix}
0&0&-1\\
0&0&0\\
1&0&0
\end{pmatrix},
\end{eqnarray*} with $\alpha:=\arctan(\sqrt{M_1^2+M_2^2}/E)$, $\beta := \arctan(M_1/M_2)$, and $k:=2\sqrt{E^2+M_1^2+M_2^2}$.\smallskip

{\bf Normalizations.} To simplify the notation, we 
normalize $k=1$. This normalization corresponds to a time re-parameterization. More precisely, if $T$ is the minimum time to steer the state one to a target state for the system with $k=1$, the corresponding minimum time for 
the original system is $\displaystyle\frac{T}{2\sqrt{E^2+M_1^2+M_2^2}}$. \medskip

{\bf Assumptions.} Two types of assumptions on the parameters $\alpha$ and $\beta$ are used in this paper: 
\begin{itemize}
\item[(A1)] $0 < \alpha < \pi/4$ and $0 < \beta \leq \pi/4$.
\item[(A2)] $\alpha$ small and $\beta=\pi/4$.
\end{itemize} Assumption (A1) is used in Sec. \ref{sec:ext}. Assumption (A2) is used in Sec. \ref{s-small} and Sec. \ref{s:ext}. Note that in (A1) it is not restrictive to assume $0<\beta\leq\pi/4$, as the controls $u_1$ and $u_2$ play a symmetric role;  in (A2) $\beta=\pi/4$ corresponds to the case where the two controls have the same bound ($M_1=M_2$). 
\begin{remark}
Roughly speaking, the parameter $\alpha$ measures the relative strength of the control fields compared to the static one. The parameter $\beta$ characterizes the relative strength between the two control fields. In spin experiments, the static field represented by the drift term is many orders of magnitude larger than the radio-frequency control fields \cite[Chap. 10]{lev}. Therefore, the most relevant case corresponds to small $\alpha$. %Some results will only be derived for small $\alpha$ and $\beta={\pi}/{4}$. 
\end{remark}

%%%%%%%%%%%%%%%%%%%%%%%%%%%%%%%%%%%%%%%5

The vector fields $Fx$, $G_1x$, and $G_2x$ describe rotations respectively
around the axes $x_3$, $x_1$, and $x_2$. The state one which corresponds to the lowest energy level is represented by the point $N:=(0,0,1)$ (called {\it north pole}) and the state two which corresponds to the highest energy level is represented by the point $S:=(0,0,-1)$ 
(called {\it south pole}). { The optimal control problem we are interested in is  to 
connect the north pole to  any other fixed state}  in minimum time. { The most important final state is of course the south pole. In the case of a spin $1/2$ particle the later case corresponds to a \emph{complete spin flip}. As usual we assume control $u_i(\cdot)$ to 
be a measurable function satisfying $\vert u_i(t)\vert\leq 1$ almost everywhere. 
The  corresponding trajectory is a Lipschitz continuous 
function $x(\cdot)$ satisfying \eqref{cs-p} almost everywhere. 
Since  \eqref{cs-p} is controllable, and the set
of velocities $V(x):=\{Fx+u_1G_1x+u_2G_2x,~\vert u_1\vert\leq 1,~\vert u_1\vert\leq 1\}$ is compact and convex, solutions to the time-optimal control problem exist. \cite[Chap. 10]{agra-book}. By solution we mean an {\em optimal synthesis}, i.e., the collection of time-optimal trajectories starting from the north pole: 
$$
\{ \gamma_{\bar x}|~ \gamma_{\bar x} \mbox{ is time optimal between $N$ and } \bar x \}_{\bar x\in S^2}.
$$ 
Optimal syntheses are considered as the right concept of solutions for optimal control problems, see \cite{piccoli-sussman}. One of the most important tools for the construction of optimal synthesis is the Pontryagin Maximum Principle (PMP for short, see \cite{pont-book}, \cite[Chap. 12]{agra-book}). It is a first order necessary condition for optimality that allows us to restrict the 
set of candidate optimal trajectories. One then needs to select the optimal ones from this set. In general the selection step is the most difficult \cite{libro,piccoli-sussman}.

For arbitrary values of $\alpha$ and $\beta$ satisfying Assumption (A1), we are mainly concerned with the fist step, i.e., the construction of a restricted set of candidate time-optimal trajectories. The most difficult task is to analyze the role of singular trajectories. 
For small $\al$ and $\beta=\pi/4$ (i.e with controls bounded on the square:  $|\Omega_i|\leq M$, $i=1,2$), we complete the time optimal synthesis  up to  a neighborhood of order $\al$ of the south pole. More precisely, the optimal synthesis is composed of four families of 
trajectories starting from $N$ with $(u_1=1,u_2=1)$, or $(u_1=1,u_2=-1)$, or $(u_1=-1,u_2=-1)$, or $(u_1=-1,u_2=1)$, and switching for the first time at $s$ with $s\in[0,s_{\max}]$ and then every $v(s)$ until reaching a neighborhood of $S$. Here \emph{switching} means that one of the two controls switches from $+1$ to $-1$ or vice versa. The expressions of $s_{\max}$ and $v(s)$ are given by 
\begin{eqnarray*}
s_{\max}&=&\arccos\left(-\frac{\sin^2\alpha}{1+\cos^2\alpha}\right),\\
v(s)&=&\arccos{\left[\frac{d-A(s)-B(s)-C(s)}{e-A(s)+B(s)}\right]},
\end{eqnarray*} where $A(s): = 8 \cos\alpha \sin^2\alpha \sin(s)$, $B(s) = 2 \sin^2{2\alpha}\cos(s)$, $C(s): = 4 \sin^4\alpha\cos(2s)$, $d:= \sin^2{2\alpha}$, and $e:= 5+2\cos{2\alpha}+\cos{4\alpha}$. See Proposition \ref{bb} and Corollary \ref{coro:vsequal} for more detail. An image of the time optimal synthesis is presented in Fig. \ref{fig:optsyn}, where the non-intersecting \emph{four-snake} structure of the four families of optimal trajectories is preserved outside a neighborhood of the south pole. The colored curves called \emph{switching curves} are the locations where a switching occurs. Note that the switching curves are located close to the two great circles passing through the north pole and containing $x_1-$ or $x_2-$axis; the endpoints of each switching curve are located exactly on these great circles.
\begin{figure}[!ht]
\centering
\includegraphics[height=120mm,width=155mm]{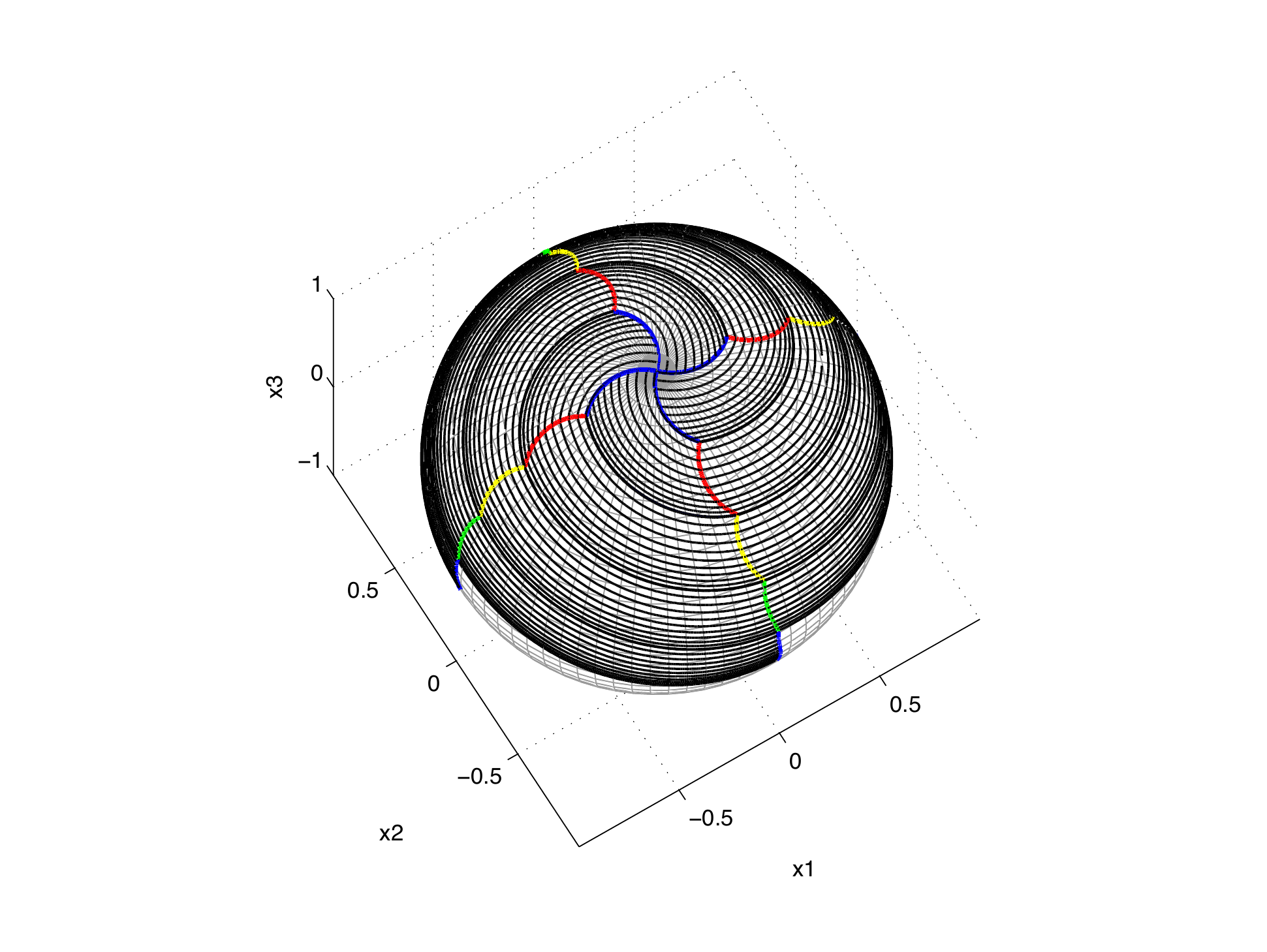}%{opt-switch}
\caption{Optimal synthesis for $\alpha=0.25$ and $\beta=\pi/4$. Black curves are optimal trajectories and colored curves are switching curves.}
\label{fig:optsyn}
\end{figure}
We deduce from the optimal synthesis that for a given target state, the optimal trajectories are bang-bang, and the corresponding optimal controls are periodic on all interior bang arcs. In other words, $u_1$ and $u_2$ are periodic except on the first and last pieces. The first and the last switching times need to be computed numerically depending on the target. These optimal  trajectories are  more complicated than the ones with the two controls bounded on the circle  (i.e. with $\sqrt{\Omega_1^2+\Omega_2^2}\leq M$, see  \cite{daless,q2}), but permit faster transfer times. We show in Fig. \ref{fig:optu} a time-optimal control that steers \eqref{cs-p} from $N$ to $S$ with $\alpha=0.25$.
\begin{figure}[!ht]
\centering
\includegraphics[height=80mm,width=80mm]{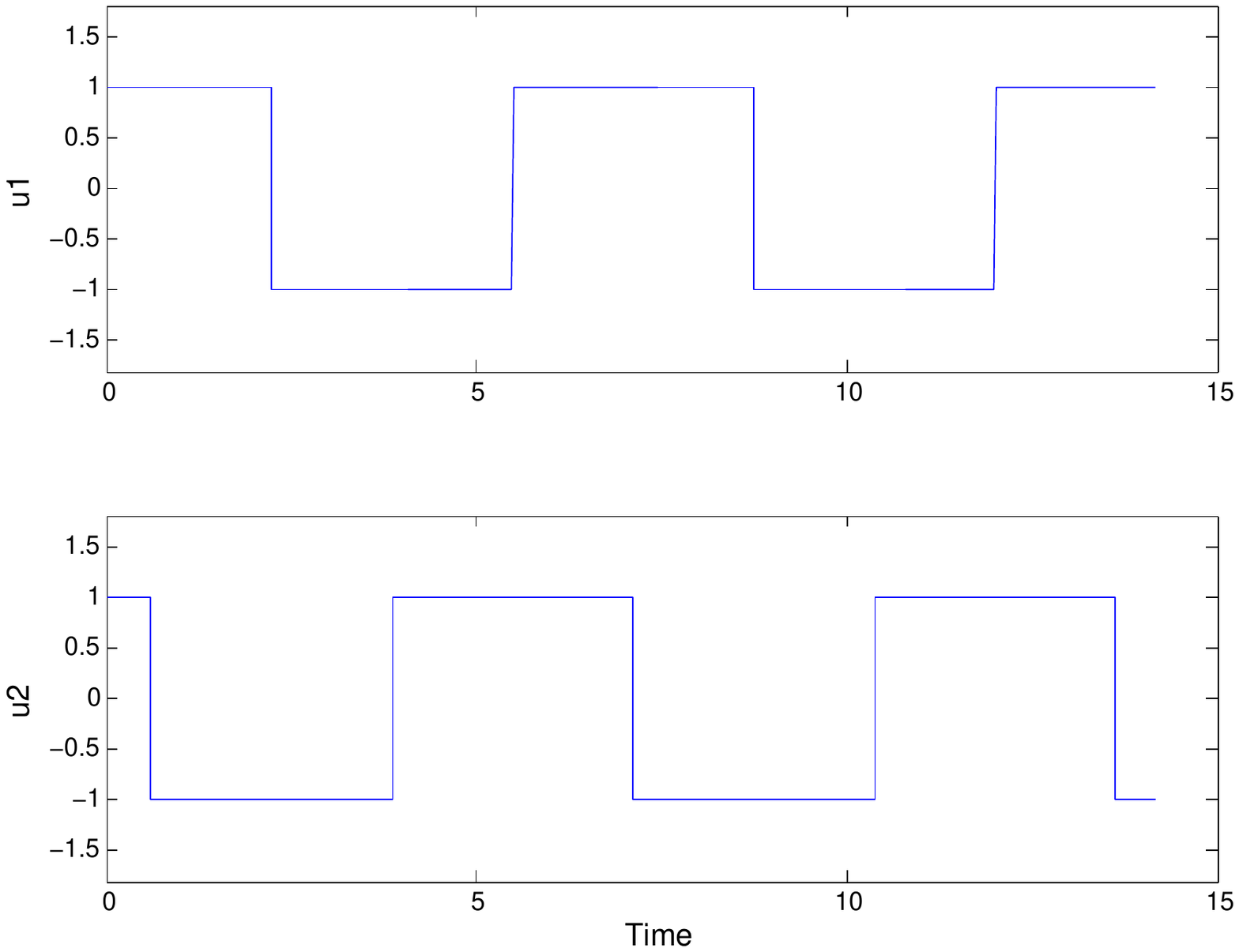}
\includegraphics[height=80mm,width=80mm]{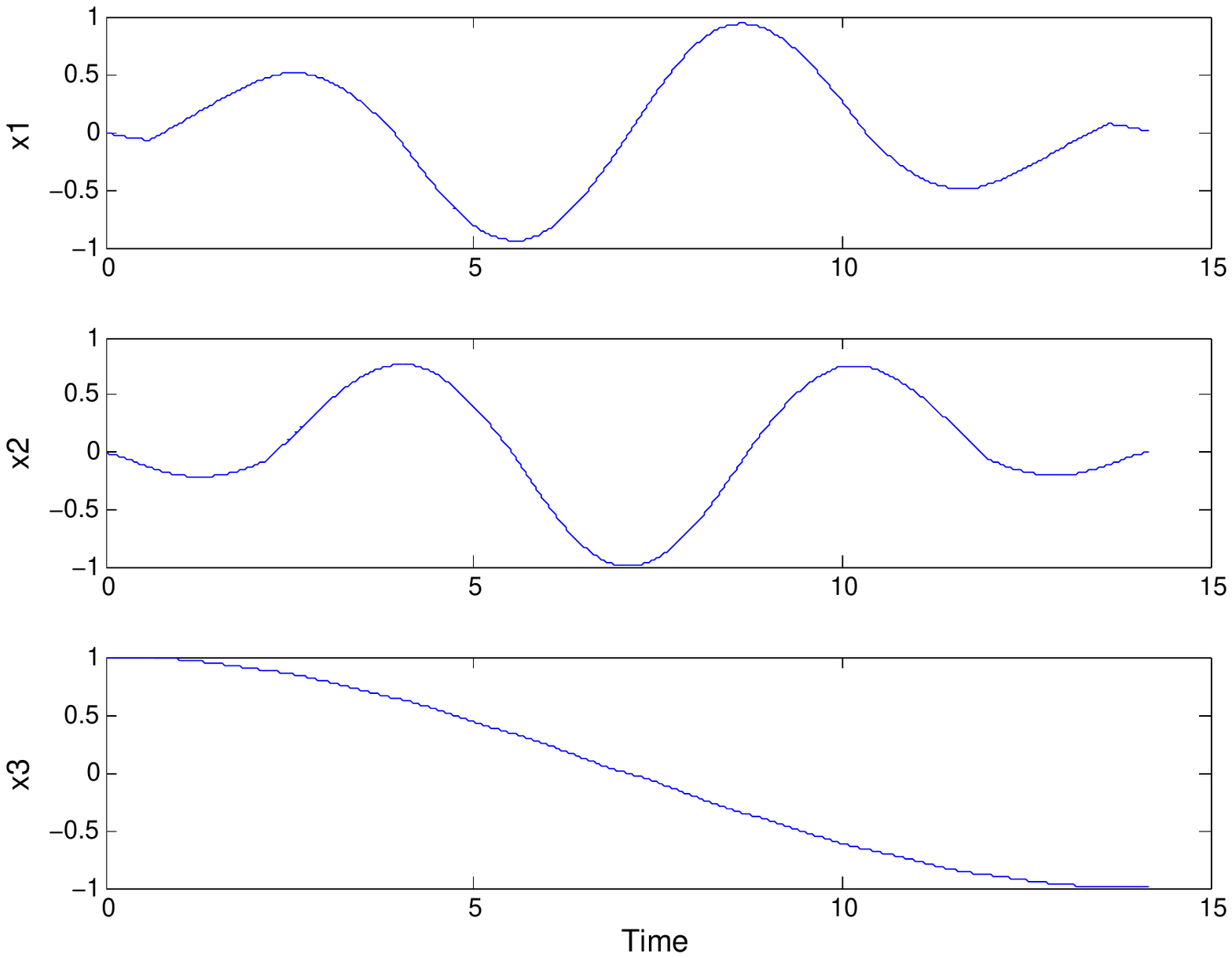}
\caption{For $\alpha=0.25$ and $\beta=\pi/4$, optimal control $(u_1, u_2)$ (left) and corresponding optimal trajectory reaching the south pole (right).}
\label{fig:optu}
\end{figure}

Based on the optimal synthesis up to a neighborhood of $S$, we also propose two families of simple suboptimal controls which both allow a transfer from $N$ to $S$ faster than the optimal controls bounded on the circle. More specifically, consider the sequence of controls
$(u_1=1,u_2=-1)\rightarrow (u_1=-1,u_2=-1)\rightarrow (u_1=-1,u_2=1)\rightarrow (u_1=1,u_2=1)$, or any of the three other cyclic permutations of it, where each pair of controls lasts for a duration equal to $\pi/2$. We show that successively applying this sequence steers system \eqref{cs-p} from $N$ to a point close to $S$ with an error of order $\alpha$. Furthermore, a similar but slightly non-saturate sequence $(u_1=\gamma,u_2=-\gamma)\rightarrow (u_1=-\gamma,u_2=-\gamma)\rightarrow (u_1=-\gamma,u_2=\gamma)\rightarrow (u_1=\gamma,u_2=\gamma)$, where $\gamma$ is an explicitly computable positive constant slightly smaller than $1$, will bring \eqref{cs-p} from $N$ to $S$ exactly.  Both strategies realize a transfer time close to the optimal one without any computation.

The paper is organized as follows. In Section \ref{sec:gr}, we derive basic facts of optimal syntheses on 2-D manifolds for control-affine systems with two bounded controls. The section is self-contained, and has its own value beyond the optimal control problem considered in this paper. Based on these results, we present in Section \ref{sec:ext} a restricted set of candidate optimal trajectories for the case with $0<\alpha<\pi/4$ and $0<\beta\leq \pi/4$. In Section \ref{s-small}, we complete the time optimal synthesis up to a neighborhood of order $\alpha$ of the south pole for small $\alpha$ and $\beta=\pi/4$. Further, we derive in Section \ref{s:ext} two simple suboptimal strategies, and compare them with the optimal strategy for controls bounded on the circle. Finally, we gather in Appendices \ref{pf:bb}, \ref{app:v} and \ref{app:com} technical proofs and computational lemmas.

\section{Optimal syntheses on 2-D manifolds with two bounded controls}\label{sec:gr}

In this section, we introduce important definitions and develop basic facts about optimal
syntheses on 2-D manifolds for control-affine systems with two bounded
controls. We use ideas similar to those used by Sussmann, Bressan,
Piccoli and the first author in \cite{sus1,sus2,bressan, libro,BCC}. This section is written to be as self-contained as possible.
%%%%%%%%%%%%%%%%%%%%%%%%%%%%%%%%%%%%%%%%%%%%%%%%%%%%%%%%
\subsection{Basic Definitions and PMP}
%%%%%%%%%%%%%%%%%%%%%%%%%%%%%%%%%%%%%%%%%%%%%%%%%%%%%%%%
We focus on the following problem:
%%%%%%%%%%%%%%%%%%%%%%%%%%%%%%%%%%%%%%%%%%%%%%%%%%%%%%%%

\medskip\noindent
{\bf (P)} {\it
Consider the control system  
\begin{eqnarray} \dot x &=& F(x)+u_1G_1(x) + u_2G_2(x),\label{mt}
\end{eqnarray}
where $x\in M$, $\vert u_i\vert\leq 1$, $i=1,2$. We make the following assumption:
\begin{description}
\item[(H0)] $M$ is a smooth 2-D manifold. The vector fields 
$F$, $G_1$ and $G_2$ are 
$\con^\infty$, and the control system (\ref{mt}) is complete on $M$. 
\end{description}
The goal is to \underline{reach every point 
of $M$ in minimum time} from a source $M_{in}$ which is assumed to be a smooth 
submanifold of  
$M$, possibly with a smooth boundary.}

In the following we use the notation $u:=(u_1,u_2)$, and $x:=(x_1,x_2)$ in a local chart.
\begin{definition}
A control for the system (\ref{mt}) is a measurable function 
$u(\cdot)=(u_1(\cdot),u_2(\cdot)):[a_1,a_2]\to 
[-1,1]^2$. The 
corresponding trajectory is a Lipschitz continuous map
$x(\cdot):[a_1,a_2]\to M$ such that $\dot 
x(t)=F(x(t))+u_1(t)G_1(x(t))+u_2(t)G_2(x(t))$ for
almost every $t\in [a_1,a_2]$. Since the system is autonomous we 
can always
assume that $[a_1,a_2]=[0,T]$. 
\end{definition}
A solution to problem {\bf (P)} is an \underline{optimal synthesis} 
that is a  
collection $\{(x_{\bar x}(\cdot), u_{\bar x}(\cdot))$ defined on $[0,T_{\bar 
x}], \bar x\in M\}$ of trajectory--control pairs
such that $x_{\bar x}(0)\in M_{in}$, $x_{\bar x}(T_{\bar x})=\bar x$, and 
\underline{$x_{\bar x}(\cdot)$ 
is time 
optimal}.

%%%%%%%%%%%%%%%%%%%%%%%%%%%%%%%%%%%%%%%%%%%%%%%%%%%%%%%%%%%%%%%%%%%%%%%

We use the following definition to describe different types of controls. 
\begin{definition}
Let $u(\cdot)=(u_1(\cdot),u_2(\cdot)):[a_1,a_2]\subset[0,T]\to[-1,1]^2$ be a 
control for 
the 
control system (\ref{mt}).
\begin{itemize} 
\item $u(\cdot)$ is a \underline{bang control} if for almost
every  $t\in[a_1,a_2]$, $$u(t)=\bar u\in 
\{(-1,-1),(-1,1),(1,-1),(1,1)\}.$$ Similarly, $u(\cdot)$ is a \underline{$u_i$-bang control} ($i=1,2$) if for almost
every  $t\in[a_1,a_2]$, $u_i(t)=\underline u\in\{\pm1\}.$

\item A $u_i(\cdot)$-switching  ($i=1,2$) is a time $\bar t\in[a_1,a_2]$ such that for a sufficiently small $\eps>0$, $u_i(\cdot)$ is a.e. equal to $+1$ on $]\bar t-\eps,\bar t[$ and a.e. equal to $-1$ on $]\bar t,\bar t+\eps[$ or vice-versa. A $u_1$-$u_2$-switching is a time $\bar t$ that is a $u_1$- and a $u_2$-switching.

\item  If $u_A:[a_1,a_2]\to[-1,1]^2$ and 
$u_B:[a_2,a_3]\to[-1,1]^2$ are 
controls, their \underline{concatenation} $u_B\ast u_A$ is the
control
$$
(u_B\ast u_A)(t):=\left\{\begin{array}{l}
u_A(t)\mbox{ for } t\in[a_1,a_2],\\
u_B(t)\mbox{ for } t\in]a_2,a_3].\end{array}\right.
$$ 
The control $u(\cdot)$ is called  \underline{bang-bang} if 
it is a finite  concatenation of bang arcs. Similarly one defines 
\underline{$u_i$-bang-bang} controls.

\item A trajectory of (\ref{mt}) is a \underline{bang trajectory} (resp. \underline{bang-bang trajectory}) if it corresponds
to a bang control (resp. bang-bang control). Similarly, one defines 
$u_i$-bang and $u_i$-bang-bang trajectories.

\end{itemize}
\end{definition}
%%%%%%%%%%%%%%%%%%%%%%%%%%%%%%%%%%%%%%%%%%%%%%%%%%%%%%%%%%%%%%%

Given two vector fields $X$ and $Y$, consider the following function
\begin{equation*} 
\Delta(X,Y)(x):=\Det~(X(x)),Y(x)),\quad x\in M,
\end{equation*}
and the set of its zeros 
\begin{equation*} 
Q(X,Y):=\{x\in M~ \textrm{s.t.}~ \Delta(X,Y)(x)=0\}.
\end{equation*}
Notice that the definition of $\Delta(X,Y)$ depends  on the choice of the coordinate 
system, but not the set $Q(X,Y)$ that is the set of points where $X$ and
$Y$ are parallel. For problem {\bf (P)}, the sets $Q(G_1,G_2)$, 
$Q(G_1,[F+G_2,G_1])$, $Q(G_1,[F-G_2,G_1])$, $Q(G_2,[F+G_1,G_2])$, $Q(G_2,[F-G_1,G_2])$, and $Q(F,G_1,G_2):=Q(F,G_1)\cap Q(F,G_2)\cap Q(G_1,G_2)$ are fundamental to the construction of the optimal 
synthesis by applying PMP, as will be seen in the next paragraphs. In fact, assuming these sets to be embedded one-dimensional submanifolds of $M$,
we have the following:
\begin{itemize}

\item $u_1$-$u_2$ switchings can only occur on the set $Q(G_1,G_2)$. See Lemma \ref{l-elle}.

\item The support of  $u_1$-$u_2$ singular trajectories (called totally singular trajectories in the following)  is always contained in the set $Q(G_1,G_2)$. See Lemma \ref{le:tts}.

\item The support of  $u_1$-singular trajectories (that are trajectories for 
which the $u_1$-switching function identically vanishes, and for which 
$u_1$ can assume values different from $\pm1$, see next section for detail) is always contained in the set $Q(G_1,[F+G_2,G_1])\cup Q(G_1,[F-G_2,G_1])$. A similar statement holds for $u_2$-singular trajectories. See Lemma \ref{le:uis}.

\item Under certain conditions,  one proves that $u_1$ can switch only once on a connected component of $M\setminus( Q(F,G_1)\cup Q(G_1,G_2)\cup Q(G_1,[F,\pm G_2,G1])\cup Q(G_2,[F,\pm G_1,G2]))$. A similar statement holds for $u_2$. See Proposition \ref{cap2:s1}.

\end{itemize}

%%%%%%%%%%%%%%%%%%%%%%%%%%%%%%%%%%%%%%%%%%%%%%%%%%%%%%%%%%%%%%%%%%%%%
For problem {\bf (P)}, Pontryagin Maximum Principle  says the following:
\begin{corollary}\label{pmp}
Consider the  control system (\ref{mt}) subject to (H0).
For every $(x,\lam,u)\in T^\ast M\times [-1,1]^2$, define
$$\HHH(x,\lam,u):=\langle\lam,F(x)\rangle+u_1\langle\lam,G_1(x)\rangle+u_2\langle\lam,G_2(x)\rangle+\lam_0.$$
If the pair $(x(\cdot),u(\cdot)):[0,T]\to M\times [-1,1]\times [-1,1]$
is time optimal, then there exist a \underline{never vanishing} 
Lipschitz continuous \underline{covector} 
$\lam(\cdot):t\in[0,T]\mapsto \lam(t)\in 
T^\ast_{x(t)}M$ and a constant $\lam_0\leq 0$ such that for a.e. $t\in [0,T]$:
\begin{description}
\item[i)]
$\dot x(t)=$ {\large $\frac{\partial \HHH
}{\partial \lam}$}$(x(t),\lam(t),u(t))$,
\item[ii)] $\dot \lam(t)=-${\large$\frac{\partial \HHH
}{\partial x}$}$(x(t),\lam(t),u(t))=-\langle\lam(t),(\nabla F+u_1(t)\nabla 
G_1+u_2(t)\nabla G_2)(x(t))\rangle$,
\item[iii)] $\HHH (x(t),\lam(t),u(t))=\HHH_M(x(t),\lam(t))$, ~~where~~ 
$\HHH_M(x,\lam):=\max\{\HHH(x,\lam,u): u\in [-1,1]^2\},$
\item[iv)] $\HHH_M(x(t),\lam(t))=0$,
\item[v)] $\langle\lam(0),T_{x(0)}M_{in}\rangle=0$ 
(transversality condition).
\end{description}
\end{corollary}\smallskip
%%%%%%%%%%%%%%%%%%%%%%%%%%%%%%%%%%%%%%%%%%%%%%%%%%%%%%%%%%%%%%%%%%%%%%%
%\begin{remark}\label{r-nedjma}
%In this version of PMP, $\lam(.)$ is always different from zero otherwise 
%the conditions {\bf iii}, {\bf iiii} would imply $\lam_0=0$.
%\end{remark}
\begin{definition}
The real-valued map $\HHH$ is 
called
\underline{PMP-Hamiltonian}. A trajectory $x(\cdot)$ (resp. a couple
$(x(\cdot),\lam(\cdot))$) satisfying conditions
{i)}, {ii)}, {iii)} and {iv)}  is called an \underline{extremal} (resp. 
an 
\underline{extremal
pair}).
If $(x(\cdot),\lam(\cdot))$ satisfies {i)}, {ii)}, {iii)} and {iv)}   
with $\lam_0=0$ (resp. $\lam_0<0$), then it is called an
\underline{abnormal extremal pair} (resp. a \underline{normal extremal pair}). 
\end{definition}

%%%%%%%%%%%%%%%%%%%%%%%%%%%%%%%%%%%%%%%%%%%%%%%%%%%%%%%%%%%%%%%%%%%%%%%
%%%%%%%%%%%%%%%%%%%%%%%%%%%%%%%%%%%%%%%%%%%%%%%%%%%%%%%%%%%%%%%%%%%%%%%
%%%%%%%%%%%%%%%%%%%%%%%%%%%%%%%%%%%%%%%%%%%%%%%%%%%%%%%%%%%%%%%%%%%%%%%
\subsection{Switching Functions}
\label{s-stps}
In this section we  are interested
in determining when controls switch from $+1$  to $-1$ or 
vice-versa and when they may assume values in $]-1,+1[$. Moreover we  
would like to predict which kind of switchings can happen, using 
properties of the vector fields $F$, $G_1$ and $G_2$.  A key role is played by \emph{switching functions}.
%%%%%%%%%%%%%%%%%%%%%%%%%%%%%%%%%%%%%%%%%%%%%%%%%%%%%%%%%%%%%%%%%%
\begin{definition} {\bf (Switching Functions)}
Let $(x(\cdot),\lam(\cdot))$ be an extremal pair. The corresponding
switching functions are defined as $\phi_i(t):=\langle\lam(t),G_i(x(t))\rangle$, 
$i=1,2$. For later use, we also define $\phi_0(t):=\langle\lam(t),F(x(t))\rangle$.
\label{d-sw-f}
\end{definition}
The switching functions $\phi_1$ and $\phi_2$ determine when the corresponding controls switch from $+1$ to 
$-1$ or vice-versa. In fact, from 
the maximization condition iii) of Corollary \ref{pmp}, one immediately gets:
\begin{lemma}\label{l-sw}
Let $(x(\cdot),\lam(\cdot))$ be an extremal 
pair defined on $[0,T]$ and $\phi_i(\cdot)$ the
corresponding switching functions. If $\phi_i(t)\neq 0$ for some
$t\in]0,T[$,
then there exists $\varepsilon>0$
such that $x(\cdot)$ corresponds a.e. to a constant control $u_i=\sg(\phi_i)$  on
$]t-\varepsilon,t+\varepsilon[$.
Moreover, if $\phi_i(\cdot)$ has
a zero at $t$, and if $\dot \phi_i(t)$ exists and  is strictly larger 
than
zero (resp. strictly smaller than zero) then
there exists
$\varepsilon>0$
such that $x(\cdot)$ corresponds a.e. to constant control $u_i=-1$  on
$]t-\varepsilon,t[$ and a.e. to a  constant control  $u_i=+1$  on
$]t,t+\varepsilon[$ (resp. a.e. to a constant control  $u_i=+1$  on
$]t-\varepsilon,t[$ and a.e. to  a  constant control  $u_i=-1$  on
$]t,t+\varepsilon[$).
\end{lemma}
%%%%%%%%%%%%%%%%%%%%%%%%%%%%%%%%%%%%%%%%%%%%%%%%%%%

Notice that on every
interval where
$\phi_i(\cdot)$ has no zero (resp. finitely many zeros), the
corresponding control is $u_i$-bang (resp. $u_i$-bang-bang). Another direct consequence of the PMP is:
\begin{lemma}
\label{l-elle}
Let $x(\cdot)$ be an extremal trajectory defined on $[a_1,a_2]$ and $\bar t\in]a_1,a_2[$ be an $u_1$-$u_2$-switching. Then $x(\bar t)\in Q(G_1,G_2)$. 
\end{lemma}

We are then interested in differentiating $\phi_i$. { By a simple computation one gets:}
%%%%%%%%%%%%%%%%%%%%%%%%%%%%%%%%%%%%%%%%%%%%%%%%%%%
\begin{lemma}\label{l-dif-phi}
Let $(x(\cdot),\lam(\cdot))$, defined on $[0,T]$ be an extremal pair
and $\phi_i(\cdot)$
the corresponding switching functions. Then it holds 
a.e.
\begin{eqnarray}
\dot\phi_1(t)&=&\langle\lam(t),([F,G_1]+u_2(t)[G_2,G_1])(x(t))\rangle,\label{dphi1} \\
\dot\phi_2(t)&=&\langle\lam(t),([F,G_2]+u_1(t)[G_1,G_2](x(t))\rangle,\label{dphi2}
\end{eqnarray}where $[\cdot,\cdot]$ denotes the Lie bracket of two vector fields. 
\end{lemma}\smallskip

%%%%%%%%%%%%%%%%%%%%%%%%%%%%%%%%%%%%%%%%%%%%%%%%%%%%%
From Lemma \ref{l-sw} it follows that $u_i$ can assume values different 
from $\pm1$ on some interval $[a_1,a_2]$ only if the corresponding 
switching 
function vanishes identically on this interval.
%%%%%%%%%%%%%%%%%%%%%%%%%%%%%%%%%%%%%%%%%%%%%%%%%%%%%%%%%%%%%%
\begin{remark}
\label{r-erre}
Lemma \ref{l-dif-phi} asserts that if in a neighborhood of a $u_1$-switching we have that $u_2$ is a.e. equal to $+1$ or a.e. equal to $-1$, 
then in that neighborhood $\phi_1(\cdot)$ 
is a.e. a $\con^1$ function.  A similar statement holds for 
$\phi_2(\cdot)$.
\end{remark}

%%%%%%%%%%%%%%%%%%%%%%%%%%%%%%%%%%%%%%%%%%%%%%%%%%%%%%%%%
\noindent
\subsection{Abnormal Extremals}
The following lemma is again direct consequence of the PMP. It characterizes some properties of abnormal extremals.
\begin{lemma}
Let $(x(\cdot),\lam(\cdot))$ be an abnormal extremal defined on $[a_1,a_2]$. We have:\\
1. If $\bar t$ is a $u_1$-$u_2$-switching, then  $x(\bar t)\in Q(F,G_1,G_2)$.\\
2. If  $\bar t$ is a $u_1$-switching and  $u_2$ is a.e. equal to $+1$ or a.e. equal to $-1$
in $]\bar t-\eps,\bar t +\eps[$ for some $\eps>0$, then  $x(\bar t)\in Q(F\pm G_2,G_1)$.\\
3. If  $\bar t$ is a $u_2$-switching and  
 $u_1$ is a.e. equal to $+1$ or a.e. equal to $-1$
in $]\bar t-\eps,\bar t +\eps[$ for some $\eps>0$,
  then  $x(\bar t)\in Q(F\pm G_1,G_2)$.
\end{lemma}

%%%%%%%%%%%%%%%%%%%%%%%%%%%%%%%%%%%%%%%%%%%%%%%%%%%%%%%%%
%%%%%%%%%%%%%%%%%%%%%%%%%%%%%%%%%%%%%%%%%%%%%%%%%%%%%%%%%%%

%%%%%%%%%%%%%%%%%%%%%%%%%%%%%%%%%%%%%%%%%%%%%%%%%%%%%%%%%%%%%%%%%

\subsection{Singular trajectories}

\begin{definition}
Consider an extremal trajectory $x(\cdot)$  defined on $[a_1,a_2]$.
It is called \underline{$u_i$-singular}
if the corresponding switching function
$\phi_i(\cdot)$ vanishes identically on $[a_1,a_2]$.
It is called \underline{totally singular}  if $\phi_1(\cdot)$ and $\phi_2(\cdot)$ both vanish identically on $[a_1,a_2]$.
\end{definition}

{ The following two lemmas are obtained immediately from the PMP.}

\begin{lemma}\label{le:tts}
Let $x(\cdot)$ be a totally singular trajectory 
on  $[a_1,a_2]\subset[0,T]$,   then $\textrm{Supp}(x(\cdot)|_{[a_1,a_2]})\subset Q(G_1,G_2)$.
\end{lemma}

\begin{lemma}\label{le:uis}
Let $x(\cdot)$ be a $u_1$-singular trajectory 
on  $[a_1,a_2]\subset[0,T]$ and assume that $u_2$ is a.e. equal to $+1$ (resp. a.e. equal to $-1$) on $[a_1,a_2]$.  Then $\textrm{Supp}(x(\cdot)|_{[a_1,a_2]})\subset Q(G_1,[F\pm G_2,G_1])$. Similar result holds true for $u_2$-singular trajectory.
\end{lemma}

\subsection{Predicting switchings}
\newcommand{\so}{super-ordinary}

\begin{definition}
A point $x\in M$ is called a $u_1$-\so\ point if 
$$x\notin Q(F,G_1)\cup Q(G_1,G_2)\cup Q(G_1,[F\pm G_2,G1]).$$
On the set of $u_1$-\so\ points we can define the  functions $\alpha_1(x),\beta_1(x),\omega_1(x),\xi_1 (x)$ as:
\begin{eqnarray}
 [F,G_1](x)&=&\alpha_1(x) F(x)+ \beta_1(x) G_1(x),\\\
  [G_2,G_1](x)&=&\omega_1(x) G_1(x)+ \xi_1(x) G_2(x).
\end{eqnarray}
\end{definition}
The following Lemma is not used in the rest of the paper, but is a step towards understanding systems of the form \eqref{mt}.

\begin{proposition}\label{cap2:s1}
Let $\Omega\subset M$ be an open connected set composed of 
$u_1$-\so\ points. Assume that for every $x\in\Omega$  we have $\alpha_1(x) > 0$ and $\xi_1(x)-\alpha_1(x) > 0$. (resp. $\alpha_1(x) < 0$ and $\xi_1(x)-\alpha_1(x) < 0$). Then all extremal trajectories 
$x(\cdot):[a_1,a_2]\to \Omega$,  
are $u_1$-bang-bang with at most one $-1\to +1$ $u_1$-switching  switching (resp.  $+1\to -1$ $u_1$-switching). A similar result holds for $u_2$-switchings.
\end{proposition}

\begin{proof}
Let $x(\cdot):~]a_1,a_2[\to\Omega$ be an extremal trajectory and 
$\phi_1(\cdot)$ be the corresponding $u_1$-switching function.
If $\phi_1(\cdot)$ has no zero, then $x(\cdot)$ is a $u_1$-bang and the 
conclusion follows. 
Let $\bar t$ be a zero of $\phi_1(\cdot)$. 
The time  $\bar t$ cannot be a zero of $\phi_2(\cdot)$, otherwise we would have 
$x(\bar t)\in Q(G_1,G_2)$. From Remark \ref{r-erre} it 
follows that $\phi_1(\cdot)$ is a.e.   $\con^1$ in a neighborhood of $\bar t$. Without loss of generality we assume  that $\phi_1(\cdot)$  is $\con^1$ in a neighborhood of $\bar t$. 
Moreover  $\bar t$ 
cannot be a zero of $\dot \phi_1(\cdot)$ otherwise $x(\bar t)$ could
not be a $u_1$-{super}-ordinary point (we would have $x(\bar t)\in Q(G_1,[F\pm G_2,G_1] )$).
Since in a
neighborhood of $\bar t$,
$u_2$ is a.e. constantly equal to $+1$ or $-1$, we can assume $u_2$ 
constant in this neighborhood, and we have
\begin{eqnarray}
\dot \phi_1(t)&=& \langle\lam(t),([F,G_1]+u_2(t)[G_2,G_1])(x(t))\rangle\nonumber\\
&=&\alpha_1\langle\lam(t),F\rangle+\beta_1\langle\lam(t),G_1\rangle+u_2(t)\omega_1\langle\lam(t),G_1\rangle+u_2(t)\xi_1\langle\lam(t),G_2\rangle\nonumber\\
&=&{ \alpha_1\phi_0(t)+(\beta_1+ u_2(t)\omega_1)\phi_1(t)+u_2(t)\xi_1\phi_2(t)
.}\nonumber
\end{eqnarray}
At time ${\bar t}$, $\phi_1=0$ and if  $\alpha_1>0$ and $\xi_1-\alpha_1>0$ on $\Omega$,  we have
\begin{eqnarray*}
\dot \phi_1(\bar t)=\alpha_1(\phi_0(\bar t)+u_2(t)\phi_2(\bar t))+(\xi_1-\alpha_1)u_2(\bar t)\phi_2(\bar t)~>~0,
\end{eqnarray*}
where we used the following facts: {i)} from the maximization condition the quantity $u_2(\bar t)\phi_2(\bar{t})>0$ in a neighborhood of $\bar t$; {ii)} $\lam_0\leq0$ implies that  $\phi_0(\bar{t})+u_2(\bar{t})\phi_2(\bar{t})\geq-\delta$ for some arbitrary $\delta>0$ in a sufficiently small neighborhood of $\bar t$ (depending on $\delta$). The case $\alpha_1<0$ and $\xi_1-\alpha_1<0$ on $\Omega$ is treated similarly. 
\end{proof}

\section{Properties of extremals for system \eqref{cs-p}}\label{sec:ext}
Based on the general results presented in Section \ref{sec:gr}, we derive properties of extremals for system \eqref{cs-p} under assumption (A1), i.e., $0<\al<\pi/4$ and $0<\beta\leq \pi/4$. We show in particular that  starting from the north pole, only normal bang-bang trajectories can be optimal. All the results presented in this section are essentially based on the following lemma which characterizes the time evolution of switching functions corresponding to system \eqref{cs-p}.
\begin{lemma}\label{eq:swi}
Let $\phi_0$, $\phi_1$, and $\phi_2$ be the switching functions for system \eqref{cs-p}. We have:
\begin{itemize}
\item[(i)] $\displaystyle \begin{pmatrix}
\dot{\phi}_0\\
\dot{\phi}_1\\
\dot{\phi}_2\\
\end{pmatrix}=P(u_1{ (t)},u_2{ (t)})\begin{pmatrix}
{\phi}_0\\
{\phi}_1\\
{\phi}_2\\
\end{pmatrix}$, with 
\begin{equation*}
P(u_1{ (t)},u_2{ (t)}):=\begin{pmatrix}
0&\dfrac{\cos\alpha}{\tan\beta}\,u_2{ (t)} &-\cos\alpha\tan\beta \, u_1{ (t)}\smallskip\\
-\dfrac{\sin^2\alpha\sin\beta\cos\beta}{\cos\alpha}u_2{ (t)} \quad \!& 0 & \cos\alpha\tan\beta\smallskip\\
\dfrac{\sin^2\alpha\sin\beta\cos\beta}{\cos\alpha}u_1{ (t)}&-\dfrac{\cos\alpha}{\tan\beta} \qquad&0
\end{pmatrix}.
\end{equation*}
\item[(ii)] On a bang-bang trajectory, we have $$\displaystyle \phi_0(t)+\vert \phi_1(t)\vert+\vert \phi_2(t)\vert +\lambda_0=0.$$
\item[(iii)] $\displaystyle \phi_0^2(t)+\frac{1}{\tan^2\alpha}\left(\frac{\phi_1^2(t)}{\sin^2\beta}+\frac{\phi_2^2(t)}{\cos^2\beta}\right)=K$, for all $t$, with $\displaystyle K:=\frac{1}{\tan^2\alpha}\left(\frac{\phi_1^2(0)}{\sin^2\beta}+\frac{\phi_2^2(0)}{\cos^2\beta}\right)$.
\end{itemize}
\end{lemma}\medskip

\begin{proof}
(i) is a consequence of Lemma \ref{l-dif-phi}. (ii) is a consequence of (iv) of Corollary \ref{pmp} and Lemma \ref{l-sw}. (iii) is based on (i) and the fact that $\phi_0(0)=0$.
\end{proof}

\subsection{Normal and abnormal bang-bang extremals}
\begin{proposition}\label{bb}
Under (A1), normal extremals for \eqref{cs-p} have the following properties:
\begin{itemize}
\item [(i)] Let $s>0$ and $s+t$ ($t>0$) be two consecutive switching times. If $\phi_2(s)=0$ (resp. $\phi_1(s)=0$), then $\phi_1(s)\neq0$, $\phi_1(s+t)=0$ and  $\phi_2(s+t)\neq0$ 
(resp. $\phi_2(s)\neq0$, $\phi_2(s+t)=0$ and  $\phi_1(s+t)\neq0$).
\item [(ii)] The duration of the first bang-arc $s$ satisfies $[0,s_{\max}]$ with
\begin{equation}
	s_{\max} := \left\{
	\begin{array}{ll}
		\arccos\left(-\dfrac{\sin^2\alpha\cos^2\beta}{1-\sin^2\alpha\cos^2\beta}\right) & 
		\text{{ if it corresponds to control (1,1) or (-1,-1)}}, \\
		
		\arccos\left(-\dfrac{\sin^2\alpha\sin^2\beta}{1-\sin^2\alpha\sin^2\beta}\right) & 
		\text{{ if it corresponds to control (1,-1) or (-1,1)}}. 
	\end{array} \right.
\end{equation}

\item[(iii)] The duration between two consecutive switchings is the same for all interior bang arcs (i.e., excluding the first and the last bang arcs). This duration depends only on the duration of the first bang arc. 
\end{itemize}
\end{proposition}
The proof of Proposition \ref{bb} is postponed to Appendix \ref{pf:bb}. { In the following, the duration between two consecutive switchings of interior bang arcs is denoted by $v(s)$ with $s$ the duration of the corresponding first bang arc. Point (iii) of Proposition \ref{bb} is illustrated in Fig. \ref{fig:swi2}, and the explicit expression of $v(s)$ is given in Appendix \ref{app:v}.
%{
%In this section, we illustrate Proposition \ref{bb} by numerical simulations. To fix the idea, assume that these trajectories start from the north pole with the control $u=(1,1)$.
%%To emphasize the effect of $\alpha$ and $\beta$ we show simulations with $\alpha = 0.25$ as well as $\alpha = \pi/4$, and $\beta = \pi/8$. Fig. \ref{fig:swi1} illustrates (ii) of Proposition \ref{bb}. 
%In Fig. \ref{fig:swi2} we draw the function $v(s)$, i.e. the duration between all interior switchings as a function of the first switching time, illustrating (iii) of Proposition \ref{bb}. Note that there exists a first switching time $s$ such that $v(s)$ is equal for all initial controls 
%%(in the case $\beta = \pi/4$ this is true for all $s$). 
%% 
%%\begin{figure}[!ht]
%%\centering
%%\includegraphics[height = 55mm, width = 70mm]{swi1.pdf}
%%\includegraphics[height = 55mm, width = 70mm]{swi12.pdf}
%%\caption{First switching time as a function of $\theta$ for $\alpha = 0.25$ (left) and $\alpha = \pi/4$ (right) with $\beta = \pi/8$.}
%%\label{fig:swi1}
%%\end{figure}
%%
\begin{figure}[!ht]
\centering
\includegraphics[height = 35mm, width = 50mm]{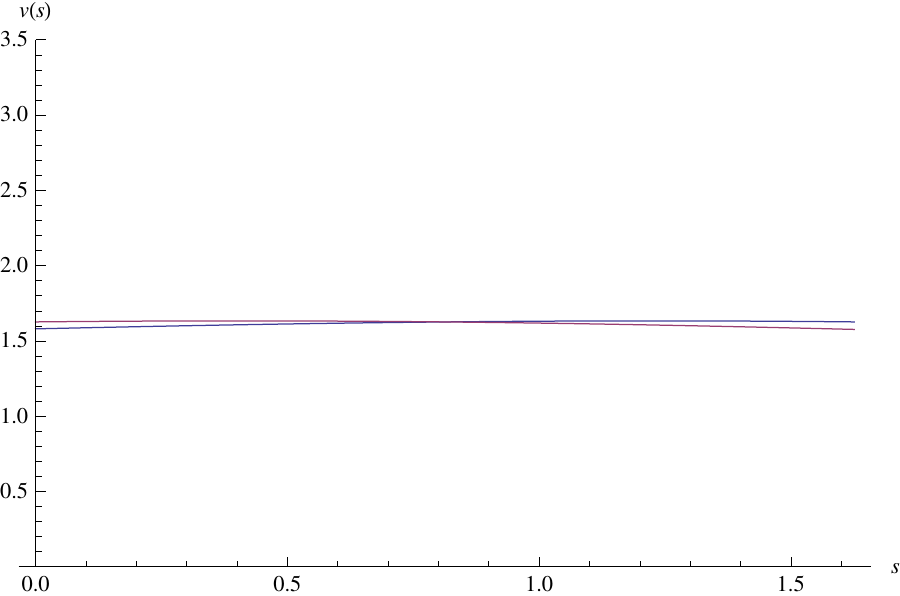}
\includegraphics[height = 35mm, width = 50mm]{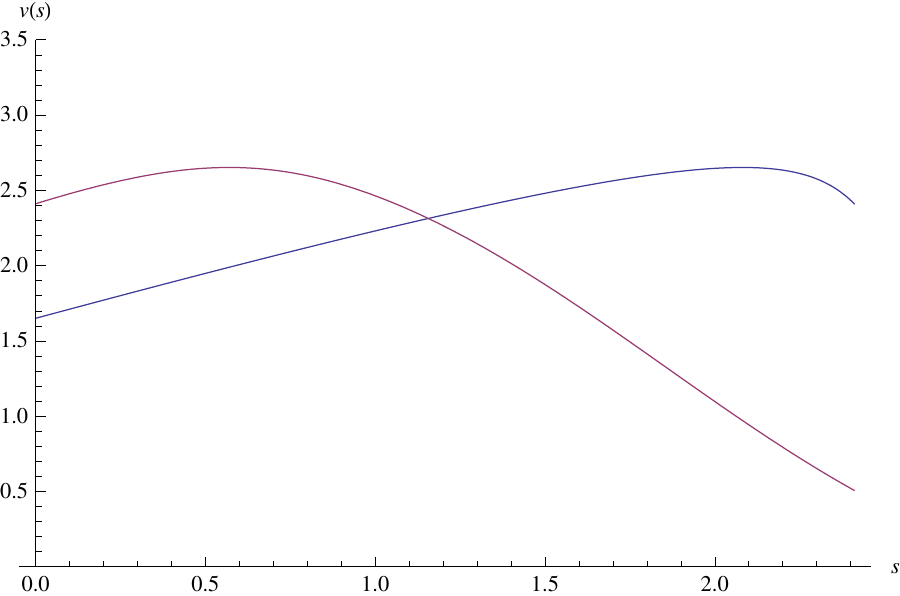}
\caption{Duration between interior switchings as a function of the first switching time $s$, for $\beta = \pi/8$ with $\alpha = 0.25$ (left) and $\alpha \approx \pi/4$ (right). The blue curve corresponds to initial control $u = \pm(1,1)$ and the red to $u = \pm(1,-1)$.}
\label{fig:swi2}
\end{figure}

\begin{corollary}\label{coro:vsequal}
If $0<\alpha<\pi/4$ and $\beta=\pi/4$, we have 
\begin{itemize}
\item [(i)] $\displaystyle s_{\max}~=~\arccos\left(-\frac{\sin^2\alpha}{1+\cos^2\alpha}\right)$.
\item [(ii)] $\displaystyle v(s)~=~\arccos{\left[\frac{d-A(s)-B(s)-C(s)}{e-A(s)+B(s)}\right]}$, where $A(s): = 8 \cos\alpha \sin^2\alpha \sin(s)$, $B(s) = 2 \sin^2{2\alpha}\cos(s)$, $C(s): = 4 \sin^4\alpha\cos(2s)$, $d:= \sin^2{2\alpha}$, and $e:= 5+2\cos{2\alpha}+\cos{4\alpha}$.
\item [(iii)] $v(0)=v(s_{\max})=s_{\max}$.
\item [(iv)]All the switching points of the extremals having their first switching at $s_{\max}$ are located on the great circles passing through $N$ and containing $x_1-$ or $x_2-$axis. 
%The extremals for which the first switching occurs at $s_{\max}$ are $\mathcal{C}^1$ after the first switching, i.e., the velocity field at switching points is continuous after the first switching. In addition, 
\end{itemize}
\end{corollary}
\begin{proof}
Point (i) is a direct consequence of Point (ii) of Proposition \ref{bb}. Points (ii)-(iv) are proved in Appendices \ref{app:v} and \ref{app:com}.
\end{proof}

The next proposition concerns abnormal extremals. It holds without Assumption (A1).
\begin{proposition}\label{prop:ab}
There are no abnormal bang-bang trajectories starting from the north pole.
\end{proposition}
\begin{proof}
Assume by contradiction that there exists an abnormal bang-bang trajectory starting from the north pole. Then, (ii) of Lemma \ref{eq:swi} implies that $$\phi_0(t)+\vert\phi_1(t)\vert+\vert\phi_2(t)\vert =0.$$ 
{ Since $\phi_0(0)=0$, 
we have $\phi_1(0)=\phi_2(0)=0$. This contradicts the non triviality of the co-vector $\lambda$.}
\end{proof} 

\subsection{Singular trajectories}
The results presented in this section characterize singular trajectories of \eqref{cs-p}. They are consequences of Lemmas \ref{le:tts}, \ref{le:uis}, and \ref{eq:swi}. The normalization used here for the co-vector $\lambda$ is given by $\lambda(0)=(\cos\theta, \sin\theta,0)$ with $\theta\in[0,2\pi[$. The corresponding initial conditions for the switching functions are:
\begin{eqnarray}
\phi_0(0)&=&0,\\
	\label{eq:phi1}\phi_1(0) &=& -\sin\alpha\sin\beta\sin\theta, \\ 
	\label{eq:phi2}\phi_2(0) &=& -\sin\alpha\cos\beta\cos\theta. 
\end{eqnarray}
With this normalization, the constant $K$ in (iii) of Lemma \ref{eq:swi} is equal to $\cos^2\alpha$. Moreover, it follows from (ii) of the same lemma that $\min_\theta \lambda_0 = -\sin\alpha$.
\begin{proposition}\label{total-sing}
The sets where the support of singular trajectories should belong to are characterized as follows.
\begin{itemize}
\item [(i)] { The support of a totally singular trajectory  must be contained in  the equator $C_0$ of $S^2$. The corresponding totally singular control satisfies   $u_1\equiv0$ and $u_2\equiv0$ almost everywhere.}
\item[(ii)] { The support of a $u_1$-singular (resp. $u_2$-singular) trajectory 
must be contained in} the set $\displaystyle C_{1\pm}:=S^2\cap\{\pm\tan\alpha\cos\beta \, x_2=x_3\}$ (resp. $\displaystyle C_{2\pm}:=S^2\cap\{\pm\tan\alpha\sin\beta\, x_1=-x_3\}$). { The corresponding $u_1$-singular (resp. $u_2$-singular) control 
satisfies a.e.  $u_1\equiv0$ and $u_2\equiv \bar u$, where $\bar u \in\{ \pm 1\}$ (resp. $u_1\equiv \bar u$ and $u_2=0$ a.e.).}
\end{itemize}
\end{proposition}\smallskip

\begin{proof}
For (i), applying Lemma \ref{le:tts}, $G_1x(t)$ must be parallel to $G_2x(t)$. Therefore, $x_3(t)=0$ on $[a,b]$, i.e., a totally singular trajectory can only stay on the equator of $S^2$. For (ii), assume for instance $\phi_1=0$ and $\phi_2\neq 0$ on some interval $[a,b]$. Applying Lemmas \ref{le:uis} and \ref{eq:swi}, $G_1x(t)$ is parallel to $\displaystyle(G_2-u_2\tan^2\alpha\cos^2\beta \, F)x(t)$, i.e.,
$$\begin{pmatrix}
0\\
-x_3\\
x_2
\end{pmatrix}\wedge\begin{pmatrix}
\displaystyle u_2\tan\alpha\cos\beta \,x_2-x_3\\
\displaystyle-u_2\tan\alpha\cos\beta \,x_1\\
x_1
\end{pmatrix}=(u_2\tan\alpha\cos\alpha \, x_2 - x_3)x=0.
$$ Therefore, a $u_1$-singular trajectory must stay in the set $\displaystyle C_{1\pm}:=S^2\cap\{\pm\tan\alpha\cos\beta\, x_2=x_3\}$. The proof for $u_2$-singular trajectory is similar.
\end{proof}

\begin{figure}[!ht]
\centering
\includegraphics[height=80mm,width=100mm]{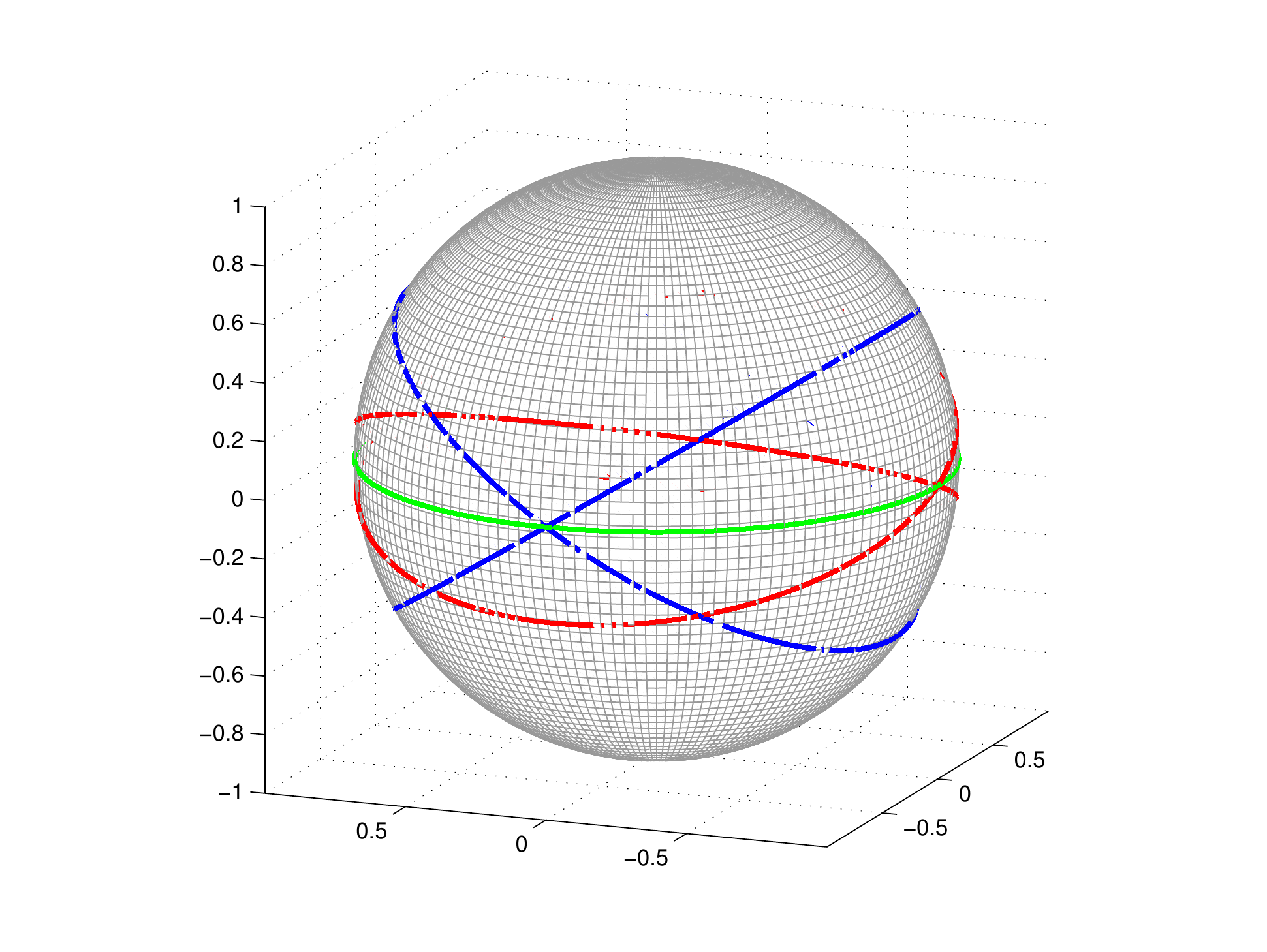}
\caption{$C_0$, $C_{1\pm}$, and $C_{2\pm}$ are represented respectively by green, red, and blue lines for $\alpha=\pi/5$ and $\beta=\pi/6$.}
\label{fig:loci}
\end{figure}

\begin{corollary}\label{bb-tsing}
Assuming (A1), a normal bang-bang extremal starting from the north pole cannot connect to a totally singular arc.
\end{corollary}
\begin{proof}
Assume by contradiction that a normal bang-bang extremal is connected to a totally singular extremal at $t_c$. Then by Lemma \ref{eq:swi}, we have
\begin{eqnarray*}
\phi_0(t_c)&=&-\lambda_0,\quad \phi_0^2(t_c)~=~\cos^2\alpha,%\frac{1}{\tan^2\alpha}\left(\frac{\phi_1^2(0)}{\sin^2\beta}+\frac{\phi_2^2(0)}{\cos^2\beta}\right),
\end{eqnarray*}
implying that 
%\begin{equation}\label{eq:alpha}
$\lambda_0^2=\cos^2\alpha$. However, we already now that 
$$\max_\theta\lambda_0^2=\sin^2\alpha<\cos^2\alpha,\quad \textrm{if } \alpha<\pi/4,$$ which yields a contradiction.
%\sin^2\beta\sin^2\theta + \cos^2\beta\cos^2\theta + \sin\beta\cos\beta\vert\sin2\theta\vert=\frac{1}{\tan^2\alpha}.
%\end{equation}
%It is easy to check that the left hand side is smaller than 1 for any $\theta$, thus if $\alpha < \pi/4$ there are no solutions and we have a contradiction. In fact, there is a constant $\alpha_c(\beta) \in [\pi/4,\arctan\sqrt2]$ such that if $\alpha < \alpha_c(\beta)$ there are no solutions to \eqref{eq:alpha}.
\end{proof}

\begin{corollary}
If $\alpha < \pi/4$, a normal bang-bang extremal starting from the north pole cannot connect to a partially singular arc. 
\end{corollary}
\begin{proof}
Assume for instance $\phi_2(t_c)=\dot\phi_2(t_c)=0$ and $\phi_1(t_c) \neq 0$. Note that if $\phi_1(t_c)=0$, then we proceed as in the proof of Corollary \ref{bb-tsing} to achieve a contradiction. From (ii) and (iii) of Lemma \ref{eq:swi}, we obtain
\begin{eqnarray*}
\phi_0(t_c) + \vert\phi_1(t_c)\vert&=&-\lambda_0,\\
\frac{\phi_0^2(t_c)}{\cos^2\alpha}+\frac{\phi_1^2(t_c)}{\sin^2\alpha\sin^2\beta}&=&1,
\end{eqnarray*}
which imply
\begin{eqnarray*}
\phi_0(t_c) &=& \dfrac{-\lambda_0\left(\sin^2\alpha\sin^2\beta - \cos^2\alpha\right)\mp\sqrt{\Delta}}{\sin^2\alpha\sin^2\beta+\cos^2\alpha}, \\
\vert\phi_1(t_c)\vert &=& \dfrac{-\lambda_0\sin^2\alpha\sin^2\beta \pm \sqrt{\Delta}}{\sin^2\alpha\sin^2\beta+\cos^2\alpha},
\end{eqnarray*}
where $\Delta := 1 - \lambda_0^2-\sin^2\alpha\cos^2\beta$. From (i) of Lemma \ref{eq:swi}, we also have
\begin{eqnarray*}
\dot\phi_2(t_c) &=& \frac{\sin^2\alpha\sin\beta\cos\beta}{\cos\alpha}u_1\phi_0(t_c) - \frac{\cos\alpha}{\tan\beta}\phi_1(t_c) \\
&=&\mp \frac{\sin^2\alpha\cos^2\beta+\cos^2\alpha}{\cos\alpha\cos\beta}\sin\beta u_1\sqrt{\Delta} ~=~0,
\end{eqnarray*}
which implies $\Delta = 0$. However, if $\alpha < \pi/$4, then $\lambda_0^2<\cos^2\alpha$ and
\begin{equation*}
\Delta > \sin^2\alpha\sin^2\beta > 0,
\end{equation*} which yields a contradiction.
\end{proof}

\section{Optimal synthesis for small $\alpha$ and $\displaystyle\beta=\frac{\pi}{4}$}
\label{s-small} 
Assuming (A2), i.e., $\alpha$ small and $\displaystyle\beta={\pi}/{4}$, we prove in this section that the extremal trajectories issued from the north pole are globally optimal until they reach a neighborhood of the south pole.
In the following we use $S^2\setminus O(\al)$ to denote the sphere $S^2$ minus a neighborhood of order $\al$ of the south pole. 

In general, proving global optimality of solutions of the PMP is not an easy task since one has to compare for each final point all extremals reaching that point.  In our case, we get the result by  a set of arguments similar to those used  for the problem with one bounded control, see \cite{B-M,limit}.  We only give a sketch of these arguments (to avoid lengthy computations similar to those made in  \cite{B-M,limit})  except for one crucial proposition that is proved in detail. These arguments are described in  the following steps.

\begin{itemize}
\item[STEP 1] We consider all extremals starting from the north pole. They are divided into 4 families depending on the value taken by the controls at the beginning, namely $(1,1)$, $(1,-1)$, $(-1,-1)$, $(-1,1)$. Let  $X_{\sg(u_1)\sg(u_2)}:=F+u_1G_1+u_2G_2$. Then, the first family of extremals has the form:
\begin{equation}\label{eq:ext}
\Xi(t, s,\alpha):=m(t,s,\alpha)\bar{M}^n(s,\alpha)e^{sX_{++}}N,\quad s\in[0,s_{\max}],
\end{equation}
where $n$ is an integer, $\bar M(s,\alpha)$ is defined by
\begin{equation}
\bar M(s,\alpha):=e^{v(s)X_{++}}e^{v(s)X_{-+}}e^{v(s)X_{--}}e^{v(s)X_{+-}},\label{eq:Mbar}
\end{equation} and $m(t,s,\alpha)$ has one of the following forms, 
\begin{equation*} 
m(t,s,\alpha)=\left\{
\begin{array}{ll}
e^{(t-\tau_1(n,s))X_{+-}},& \tau_1(n,s):=4nv(s)+s,Ê \\
e^{(t-\tau_2(n,s))X_{--}}e^{v(s)X_{+-}},&\tau_2(n,s):=(4n+1)v(s)+s,\\
e^{(t-\tau_3(n,s))X_{-+}}e^{v(s)X_{--}}e^{v(s)X_{+-}},&\tau_3(n,s):=(4n+2)v(s)+s,\\
e^{(t-\tau_4(n,s))X_{++}}e^{v(s)X_{-+}}e^{v(s)X_{--}}e^{v(s)X_{+-}},&\tau_4(n,s):=(4n+3)v(s)+s,
\end{array}
\right.
\end{equation*}
with $0<t-\tau_i(n,s)<v(s)$. Note that the integer $n$ and the function $m(t,s,\alpha)$ are determined by the target. All the three other families can be defined in a similar manner. The extremal trajectories having their first switchings at $s_{\max}$ are called ``boundary-trajectories" of the family. Each extremal of the family switches a certain number of times before reaching the south pole, and all the switching points of the different extremals form smooth curves called \emph{switching curves}. See Figure \ref{fig:4snakes}.
\begin{definition}
A (geometric) smooth curve $C$ is called a switching curve if each point of $C$ is a switching point.
\end{definition}
By construction, the following curves $C_{k}(s,\alpha)$ defined by induction are switching curves of the first family:
\begin{eqnarray}
C_{1}(s,\alpha)&=&e^{sX_{++}}N,\quad C_{k}(s,\alpha)~=~\bar M(s,\alpha)C_{k-1}(s,\alpha),\label{eq:c1}%:=e^{v(s)X_{++}}e^{v(s)X_{-+}}e^{v(s)X_{--}}e^{v(s)X_{+-}}C_{k-1,1}(s,\alpha)
\end{eqnarray}with $k>1$, $s\in[s,s_{\max}]$, and $\bar M(s,\alpha)$ defined by \eqref{eq:Mbar}.
\begin{figure}[!ht]
\centering
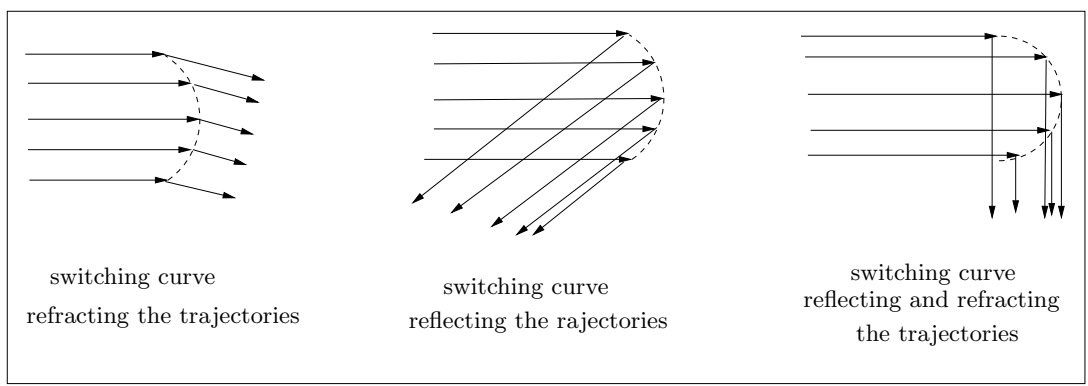
\caption{Local optimality of switching curves}
\label{fig:ref-ref}
\end{figure}

\item[STEP 2] Each switching curve can ``refract'' or ``reflect'' the extremals (see Fig. \ref{fig:ref-ref}).  The main argument of the proof is that up to a neighborhood of the south pole, all the switching curves are ``locally'' optimal, i.e., they always ``refract'' extremals.
 
 \begin{definition}
Let $C$ be a  switching curve. Assume that extremal trajectories switch on $C$ from a smooth vector field  $Y_1$ to another smooth vector field $Y_2$.
Let  $C(s)$ be a smooth parameterization of $C$ with $s\in\textrm{Domain}(C)$. We say that $C$ is \emph{locally optimal} if, for every $s\in\textrm{Domain}(C)$ and for every pair $(c_1,c_2)$ such that $c_1c_2\geq 0$, we have
$\partial_s C(s)\neq c_1Y_1(C(s))+c_2Y_2(C(s))$.
\end{definition}

\begin{proposition}\label{diofa}
Let $\alpha$ be small enough and consider the set of extremals issued from the north pole before they reach a neighborhood of the south pole. Then all the switching curves formed by this extremal flow are locally optimal.
\end{proposition}

The proof of this proposition is given in detail at the end of the section. It is clear from numerical simulations that the neighborhood of the south pole where extremal flow loses optimality is approximately a disk of radius $3\alpha$.

\begin{figure}[!ht]
\centering
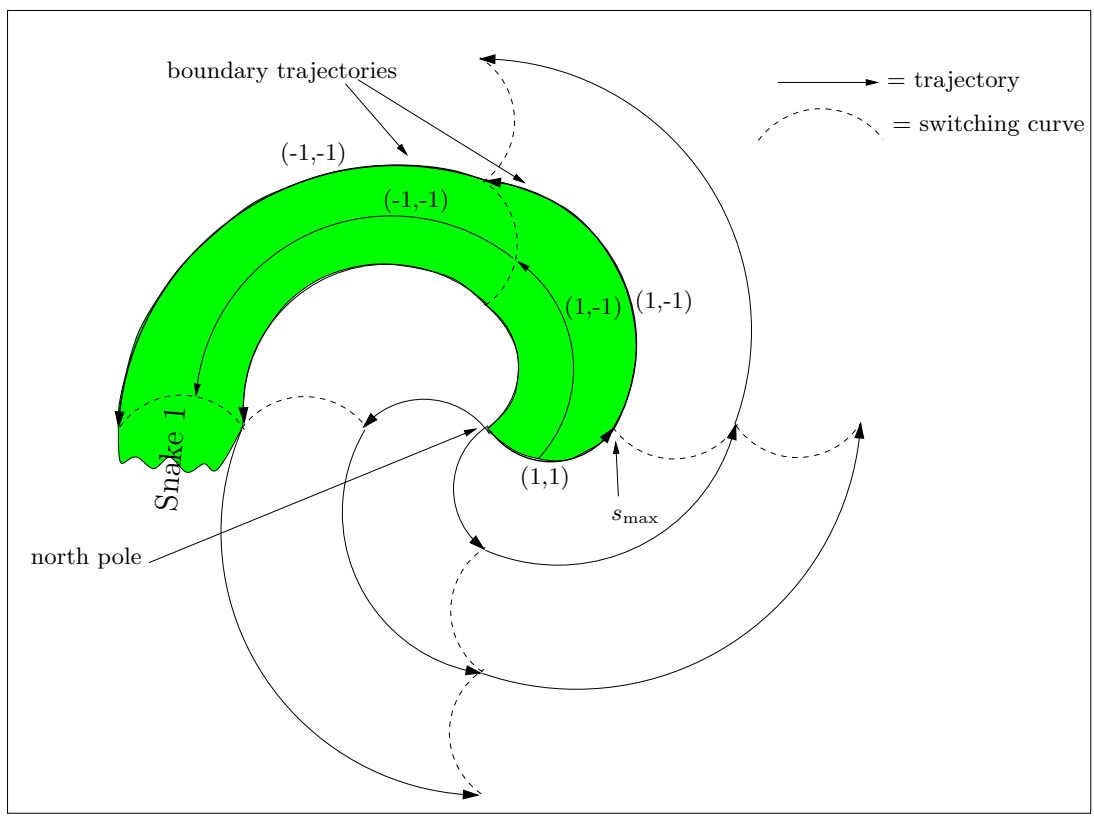
\caption{The four-snake structure.}
\label{fig:4snakes}
\end{figure}

\item[STEP 3] Proposition \ref{diofa} has two main byproducts (which are not completely obvious, but can be proved as  in \cite{B-M, limit}). 
\begin{itemize}
\item The four families of extremals defined in STEP 1 are well organized in a structure of \emph{four snakes} (see Fig. \ref{fig:4snakes} and Fig. \ref{fig:optsyn}), in the sense that the four families 
%(besides the boundary trajectories of the families) 
do not intersect until they reach a neighborhood of the south pole. 
\item   In each snake, trajectories do not intersect each other until they reach a neighborhood of the south pole.

\end{itemize}
As a consequence, each point of $S^2\setminus O(\al)$ is reached by one and only one extremal issued from the north pole before reaching a neighborhood of the south pole. By construction, these trajectories are optimal.
\end{itemize}
The \emph{four snakes} intersect in a neighborhood of order $\alpha$ of the south pole. Hence the proof fails in that region.  One can also see that in a neighborhood of the south pole there exist non locally optimal switching curves. An analysis on how the trajectories lose optimality in a neighborhood of the south pole is very complicated and out of the purpose of this paper, but it can be pursued as in \cite{limit}. The extremal front, defined as the set of the endpoints of extremal trajectories at time $t$, is homeomorphic to a circle up to a neighborhood of the south pole. In a neighborhood of the south pole it develops singularities (cusps and self-intersections) showing the presence of a cut locus (locus at which trajectories lose optimality). See Figure \ref{fig:fronts}.

\begin{figure}[!ht]
\centering
\includegraphics[height=55mm,width=55mm]{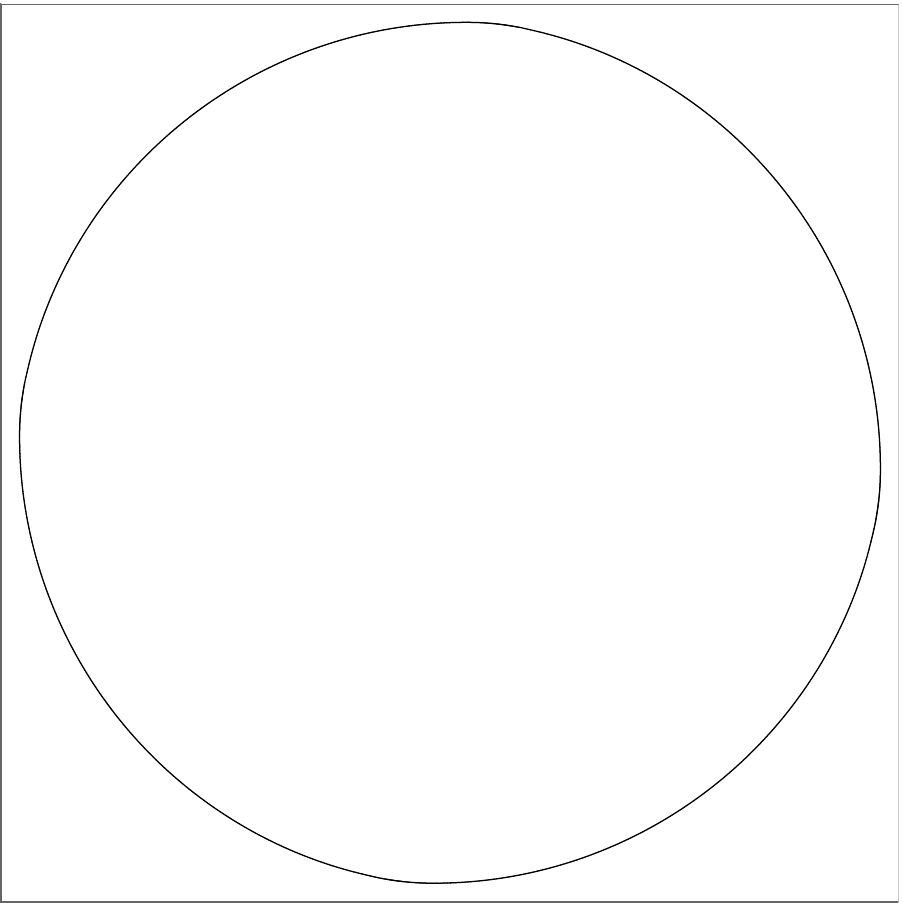}
\includegraphics[height=55mm,width=55mm]{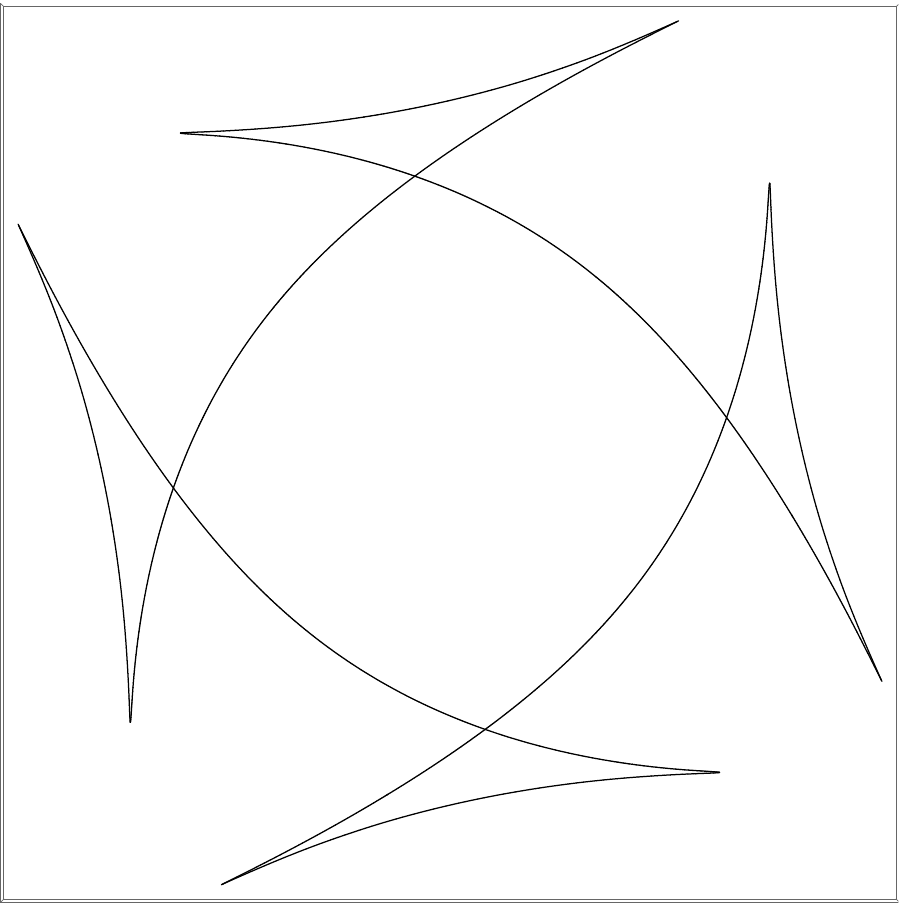}
\caption{Top view of the extremal front, far from the south pole (left) and close to the south pole (right).}
\label{fig:fronts}
\end{figure}

\begin{proof}[Proof of Proposition \ref{diofa}]
By \eqref{eq:c1}, we obtain
\begin{equation}\label{eq:switching}
C_{k+1}(s,\alpha):=M(s,\alpha)C_k(s,\alpha)=M^k(s,\alpha)C_1(s,\alpha).
\end{equation} The curves $C_k(s,\alpha)$ correspond to a switching from $(+,+)$ to $(+,-)$. Assume that there exist two real numbers $c_1$ and $c_2$ such that $c_1c_2\geq 0$, and 
\begin{eqnarray*}
\partial_sC_{k+1}&=&(c_1X_{++}+c_2X_{+-})C_{k+1}.
\end{eqnarray*}Therefore,
\begin{equation}\label{eq:local-opt}
(c_1X_{++}+c_2X_{+-})M^kC_{1}=M^{k}\partial_sC_1+\sum_{i=1}^kM^{i-1}\partial_sMM^{k-i}C_1.
%\partial_s C_1=M^{-k}(\beta_1X^{++}+\beta_2X^{+-})M^kC_{1}+M^{-k}\left(\sum_{i=1}^kM^{i-1}\partial_sMM^{-i}\right)M^kC_1
\end{equation}
By Taylor expansion, we also have
\begin{eqnarray*}
C_1(s,\alpha)&=&\begin{pmatrix}
\displaystyle \frac{\sqrt{2}}{2}(1-\sin s-\cos s)~\alpha\\
\displaystyle -\frac{\sqrt{2}}{2}(1+\sin s-\cos s)~ \alpha\\
\displaystyle1
\end{pmatrix}+O(\alpha^2),\\
\partial_sC_1(s,\alpha)&=&\begin{pmatrix}
\displaystyle \frac{\sqrt{2}}{2}(-\cos s+\sin s)~\alpha\\
\displaystyle -\frac{\sqrt{2}}{2}(\cos s+\sin s)~ \alpha\\
\displaystyle0
\end{pmatrix}+O(\alpha^2)\\
M(s,\alpha)&=&\begin{pmatrix}
1&0&0\\
0&1&-4\sqrt{2}\alpha\\
0&4\sqrt{2}\alpha&1
\end{pmatrix}+O(\alpha^2)=R(\theta)+O(\alpha^2),\\
\partial_sM(s,\alpha)&=&\begin{pmatrix}
0&4(\sin s-\cos s)~\alpha^2&0\\
-4(\sin s-\cos s)~\alpha^2&0&0\\
0&0&0
\end{pmatrix}+O(\alpha^3)=O(\alpha^2),
\end{eqnarray*} with $\displaystyle R(\theta):=\begin{pmatrix}1&0&0\\ 0&\cos\theta&-\sin\theta\\0&\sin\theta&\cos\theta\end{pmatrix}$, and $\theta:=4\sqrt{2}\alpha$. 

By \eqref{eq:local-opt}, we obtain
\begin{equation}\label{eq:local-opt2}
(c_1+c_2)\begin{pmatrix}
\sin k\theta\\
0\\
0
\end{pmatrix}=\begin{pmatrix}
\displaystyle \partial_sf_1(s)+(c_1+c_2)\cos k\theta~ f_2(s)+(c_1-c_2)\frac{\sqrt{2}}{2}\cos k\theta\smallskip\\
\displaystyle \cos k\theta~ \partial_sf_2(s)-(c_1+c_2)[f_1(s)-\frac{\sqrt{2}}{2}\cos k\theta]\smallskip\\
\displaystyle \sin k\theta~ \partial_sf_2+(c_1+c_2)\frac{\sqrt{2}}{2}\sin k\theta
\end{pmatrix}\alpha+O(\alpha^2),
\end{equation}with $\displaystyle f_1(s):=\frac{\sqrt{2}}{2}(1-\sin s-\cos s)$ and $\displaystyle f_2(s):=-\frac{\sqrt{2}}{2}(1+\sin s-\cos s)$. Eq. \eqref{eq:local-opt2} implies that 
\begin{equation}
\sin k\theta=\sin 4\sqrt{2}k\alpha=O(\alpha).
\end{equation} 
Clearly this condition can be satisfied only in a  neighborhood of order $\al$ of the north pole or of the south pole. Direct computation shows that it is not satisfied in a  neighborhood of the north pole.
It follows that the switching curves $C_{k}$ are locally optimal until intersecting a neighborhood  of order $\al$ of the south pole. All the other cases can be treated in a similar manner.
\end{proof}

\section{{Simple suboptimal controls and comparison with other strategies}} 
\label{s:ext}

\subsection{Two simple suboptimal strategies realizing complete spin flip}\label{sec:suboptimal}
Based on the optimal synthesis described in Sec. \ref{s-small} and the computational lemmas gathered in Appendix \ref{app:com}, we present in this section two simple suboptimal strategies for the case where the two controls have the same bound $M$ and the ratio $M/E$ is small, i.e., $\alpha$ is small and $\beta=\pi/4$. This case is the most relevant one for NMR applications.

\subsubsection{A really simple suboptimal strategy}
We first present the strategy for the normalized system \eqref{cs-p} with $k=1$. For small $\alpha$, we obtain from Lemma \ref{eq:vs-small-alpha} that $v(s)\approx \pi/2$. Consider the following sequence of controls 
\begin{equation*}\label{eq:sub1}
(S_1):\qquad (+1,-1)\rightarrow (-1,-1)\rightarrow (-1,+1)\rightarrow (+1,+1),
\end{equation*}
 where each combination of $(u_1,u_2)$ lasts for a duration of ${\pi}/{2}$. By Lemma \ref{eq:prod} (see also the proof of Proposition \ref{diofa}), the action the sequence $(S_1)$ produces approximately a rotation around $x_1-$axis of angle $4\sqrt{2}\alpha$. Let $\displaystyle n:=\lceil\frac{\pi}{4\sqrt{2}\alpha}\rceil$, where $\lceil K\rceil$ denotes the smallest integer not less than $K$. It is clear that $\displaystyle\frac{\pi}{4\sqrt{2}n}\leq\alpha<\frac{\pi}{4\sqrt{2}(n-1)}$. Then, starting from the north pole, applying the sequence $S_1$ for  $n$ times steers the system close to the south pole (see Fig. \ref{fig:S1}), and the error is of the order of $\alpha$.
 \begin{figure}[!ht]
\centering
\includegraphics[height=60mm,width=80mm]{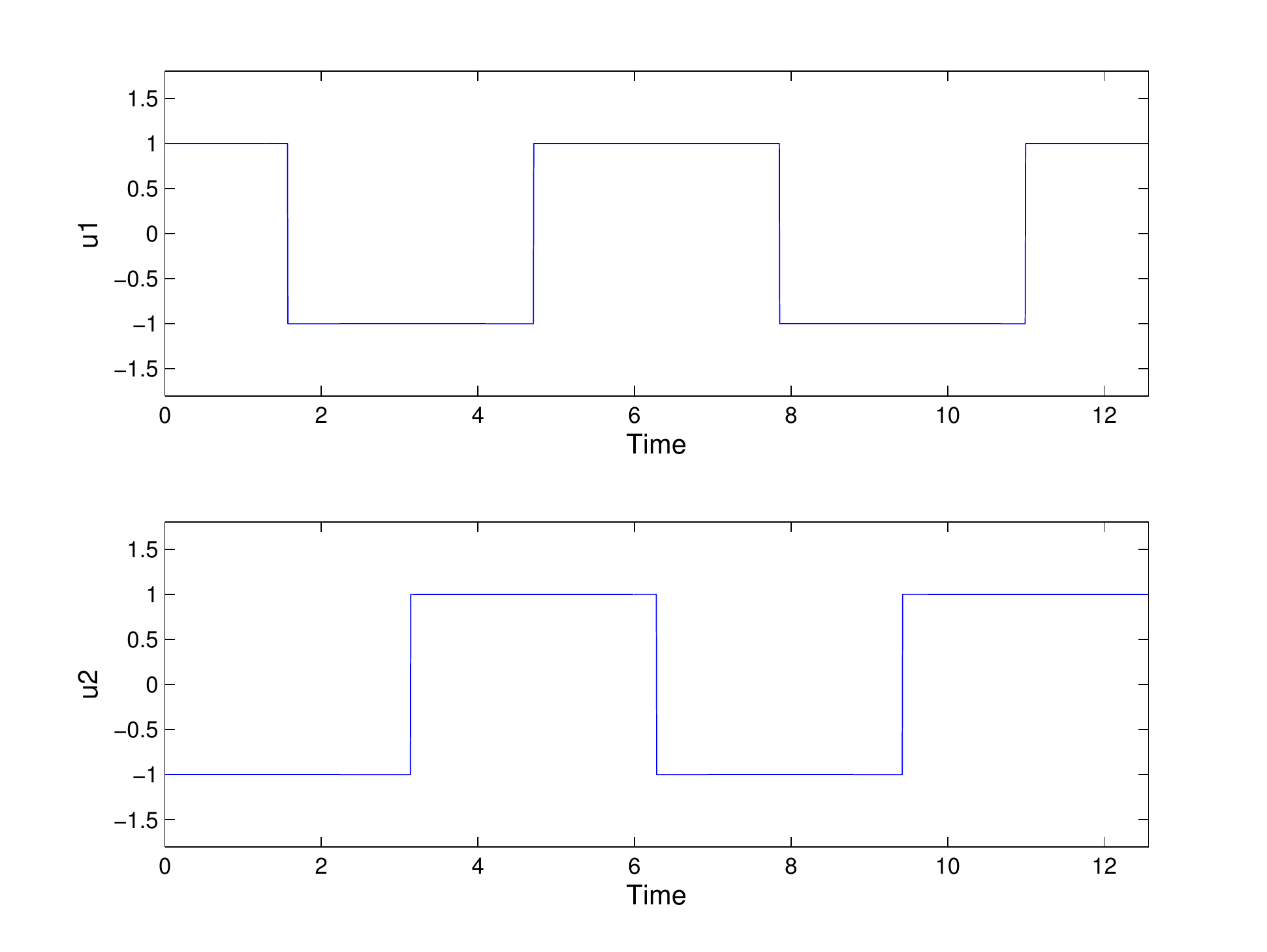}
\includegraphics[height=60mm,width=80mm]{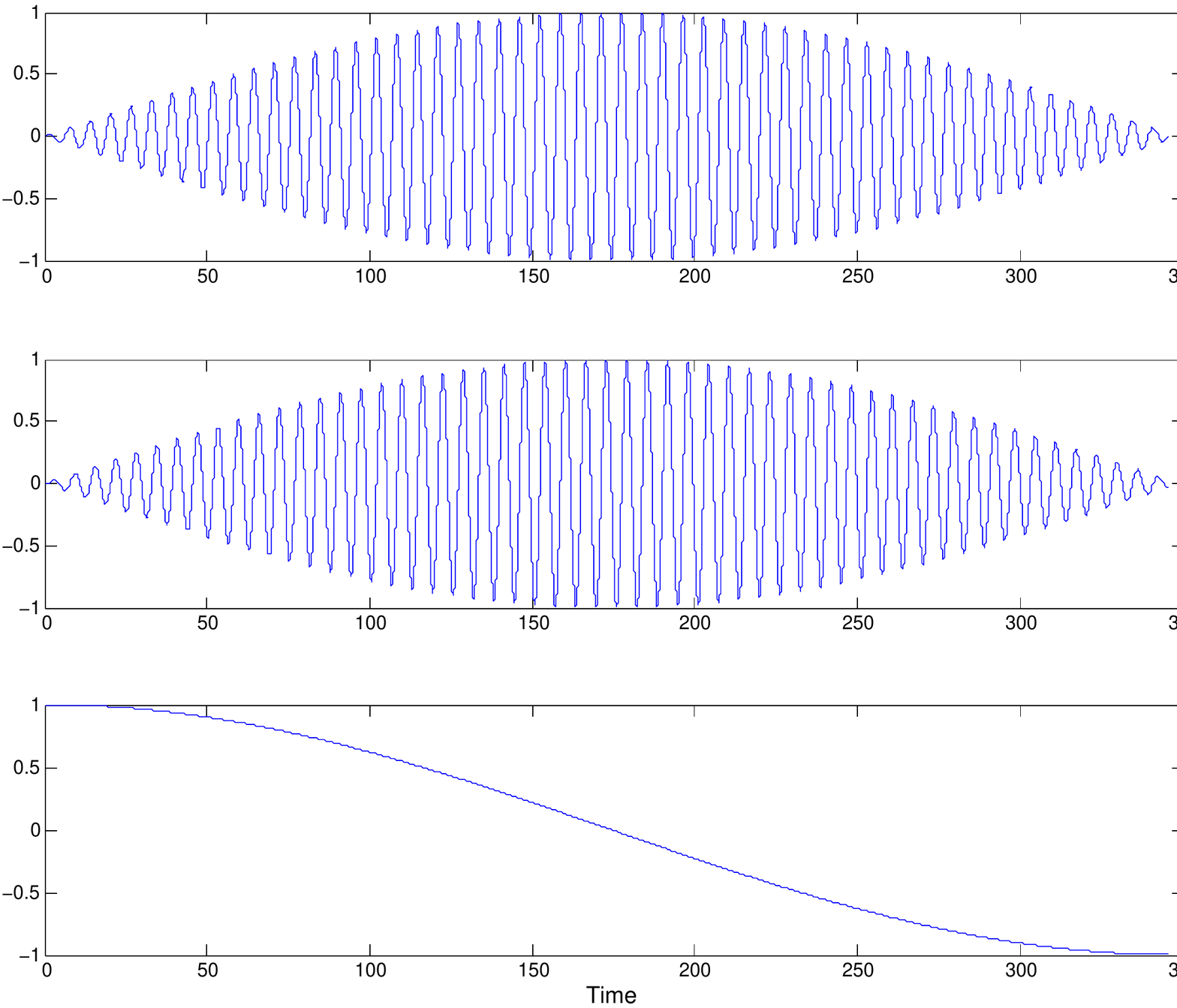}
\caption{On the left $u_1$ (top) and $u_2$ (bottom) over two periods. On the right, suboptimal trajectories for $\alpha=0.01$ and $k=1$.}
\label{fig:S1}
\end{figure}

We now take into account the time normalization constant $k$ which is approximately equal to $2E$ in the case of small $\alpha$. The suboptimal controls corresponding to the sequence $S_1$ are periodic rectangular signals of period $\pi/E$, and the total transfer time is equal to $n\pi/E$. The advantage of this strategy is that it only requires the knowledge of the Larmor frequency of the system and the bound on the control fields; it does not necessitate any computation. Note also that the three other sequences obtained from cyclic permutations of $(S_1)$ are also suboptimal in a similar manner. Finally, this strategy is suboptimal in the sense that the south pole is not exactly reached.

\subsubsection{A more accurate suboptimal strategy}\label{sec:acc-sub}
We investigate more carefully properties of the sequence $(S_1)$ to derive another suboptimal strategy that steers \eqref{cs-p} from the north pole to the south pole \emph{exactly}. The new strategy is based on the following proposition, which is a direct consequence of Corollaries \ref{lemma:Msmax} and \ref{coro:spgc}.
\begin{proposition}\label{prop:cheap2}
Let $\displaystyle \theta(\alpha):=\arcsin\left(\frac{-2\sqrt{2}\sin\alpha\cos\alpha}{1+\cos^2\alpha}\right)$, and $\bar{M}(s,\alpha)$ be defined by Eq. \eqref{eq:Mbar}. For any interger $n$ and any $\alpha<\frac{\pi}{4}$, we have
\begin{equation}
M_1^n(0,\alpha)N=\begin{pmatrix}
0\\
\sin(4n\theta(\alpha))\\
\cos(4n\theta(\alpha))
\end{pmatrix},\quad M_1^n(s_{\max},\alpha)e^{s_{\max}X_{++}}N=\begin{pmatrix}
0\\
\sin((4n+1)\theta(\alpha))\\
\cos((4n+1)\theta(\alpha))
\end{pmatrix}.
\end{equation}
\end{proposition}

In other words, Proposition \ref{prop:cheap2} states that for any $\alpha$ less than $\pi/4$ the switching curve's two endpoints stay on the great circle orthogonal to the $x_1-$axis under the action of the sequence $(S_1)$. We now construct a suboptimal strategy from the formula for $M^n_1(0,\alpha)N$. Let $\displaystyle n:=\lceil\frac{\pi}{4\theta(\alpha)}\rceil$. It is easy to check that $\theta(\alpha)$ is a monotonically increasing function of $\alpha$ for $\alpha\in[0,\pi/4]$. Then, there exists a unique $\bar{\alpha}\leq \alpha$ such that $\displaystyle 4n\theta(\bar{\alpha})=\pi$. In other words, we have
\begin{equation}
M^n_1(0,\bar{\alpha})N=\begin{pmatrix}0\\0\\-1\end{pmatrix}.
\end{equation} Let $\displaystyle \gamma:=\frac{\sin\bar{\alpha}}{\sin\alpha}\leq1$. Then, the action of $M_1(0,\bar{\alpha})$ is realized by applying the following sequence
\begin{equation*}
(S_2):\qquad (+\gamma,-\gamma)\rightarrow (-\gamma,-\gamma)\rightarrow (-\gamma,+\gamma)\rightarrow (+\gamma,+\gamma),
\end{equation*} where each combination of controls lasts for a duration equal to $\displaystyle\arccos\left(-\frac{\sin^2\bar{\alpha}}{1+\cos^2\bar{\alpha}}\right)$, which is approximately $\pi/2$. This strategy is suboptimal in the sense that the transfer time is not optimal, because the controls do not satisfy the Pontryagin Maximum Principle (the controls are bang-bang, but do not saturate the bounds $\pm1$). The advantage of this strategy is that the south pole is exactly achieved for any $\alpha$ less than $\pi/4$, and the transfer time is close to the optimal one if $\alpha$ is small. See Fig. \ref{fig:S2} for a comparison with the strategy $(S_1)$ for different values of $\alpha$. Three other similar suboptimal sequences can also be built from cyclic permutations of $(S_2)$.
 \begin{figure}[!ht]
\centering
\includegraphics[height=60mm,width=80mm]{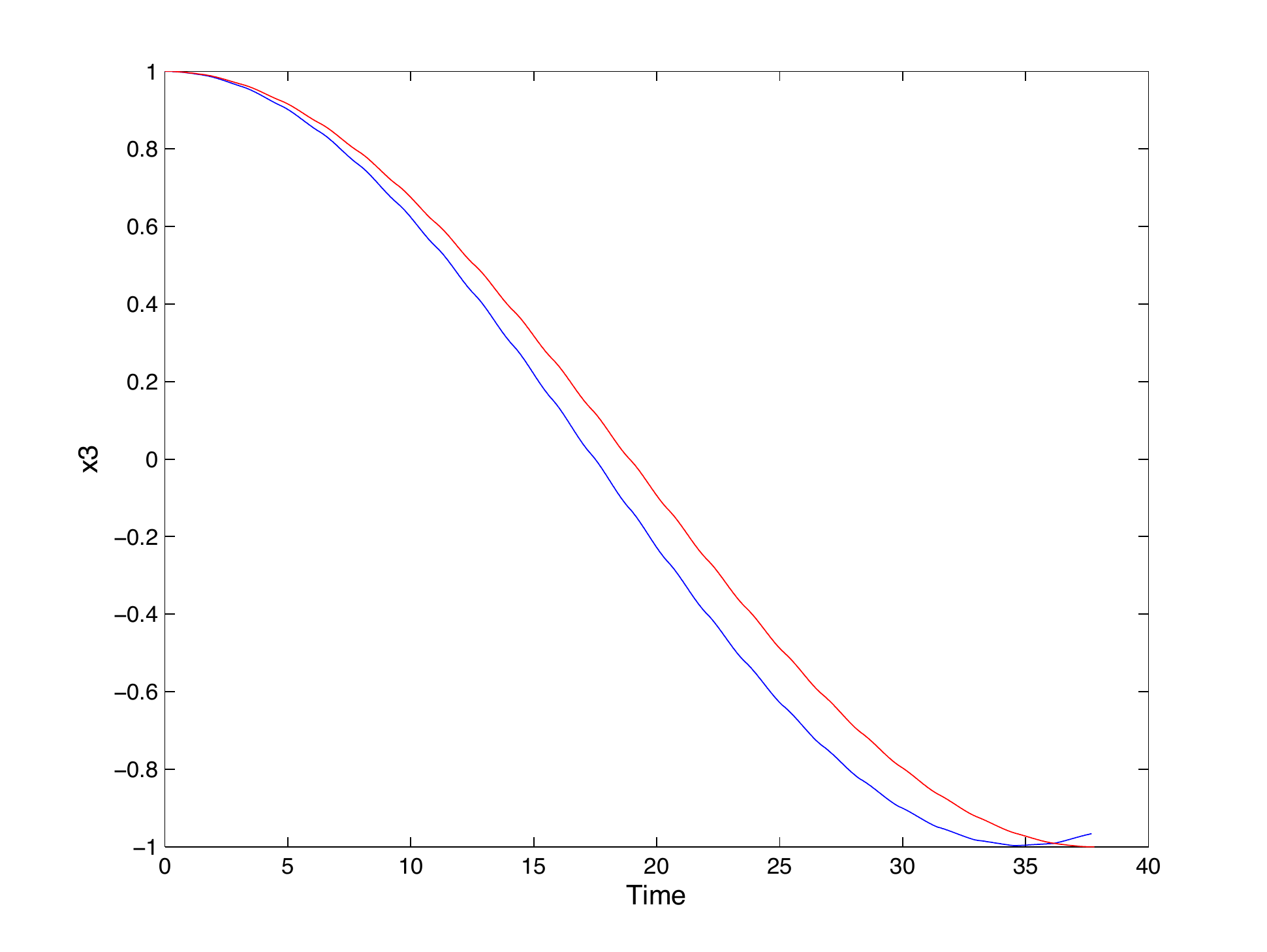}
\includegraphics[height=60mm,width=80mm]{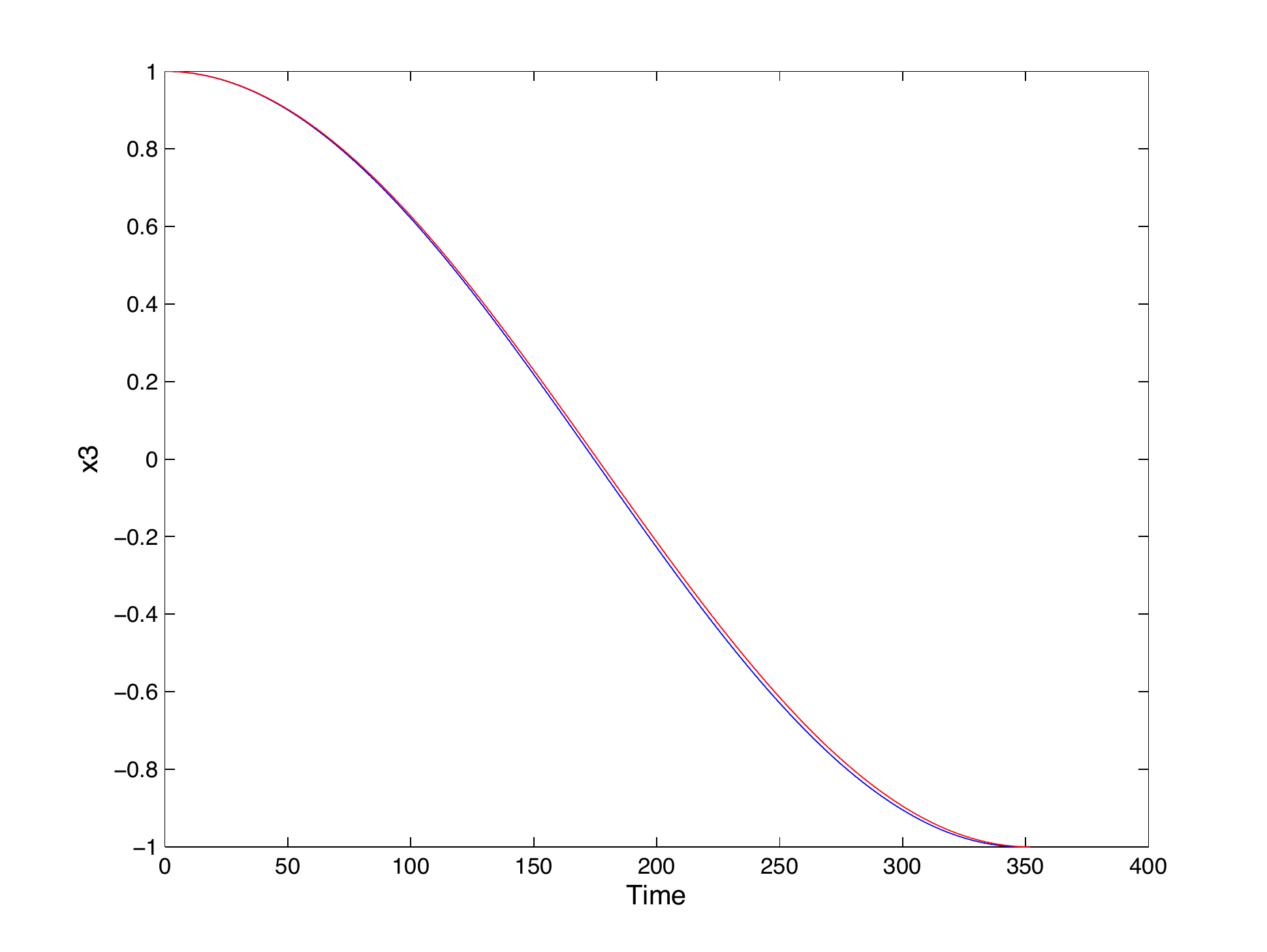}
\caption{Trajectories of $x_3$ corresponding to strategy $S_1$ (blue line) and strategy $S_2$ (red line) for $\alpha=0.1$ (left) and $\alpha=0.01$ (right).}
\label{fig:S2}
\end{figure}

}

%\section{Some results for $\alpha\geq \pi/4$}

\subsection{Comparison of suboptimal controls with optimal control bounded on the circle}\label{sec-comparison}

In this section we compare the times needed to steer system \eqref{hg1}
%i.e. in the time frame where $k = 2E/\cos{\alpha}$
from the north pole to the south pole for two different bounds on the controls:
\begin{equation*}
(S): \quad |\Omega_i(t)| \leq M, \, i = 1,2, \qquad (C): \quad \sqrt{\Omega_1^2(t)+\Omega_2^2(t)} \leq M.
\end{equation*} We are only interested in the case where $M$ is much smaller than $E$, i.e., $\alpha$ is small.
It was shown in \cite{BCC},\cite{q2} that the following control is time optimal for the control system with control bound of type $(C)$:
\begin{equation}\label{eq:circle}
\begin{split}
\Omega_1(t) &= M \sin\left(\omega_rt + \phi\right), \\
\Omega_2(t) &= M\cos\left(\omega_rt+\phi\right),
\end{split}
\end{equation}
where $\omega_r = 2E$, and that the optimal time to steer from the north pole to the south pole is given by 
\begin{equation}\label{eq:tc}
 T_C(M)=\frac{\pi}{2M}.
\end{equation}
%\begin{equation}
%T_C(M)=\frac{\pi E}{2M}=\frac{\pi}{\tan\alpha}.
%\end{equation}
We now estimate $T_S(M)$, the optimal transfer time for $(S)$. Recall that if $T_\Sigma(\alpha)$ is the optimal time for the normalized system \eqref{cs-p} with $k=1$, then we have
$\displaystyle T_S(M)={T_\Sigma(\alpha)}/{(2\sqrt{E^2+2M^2})}$. Using for instance the first suboptimal strategy presented in Sec. \ref{sec:suboptimal} (the second suboptimal strategy gives the same result up to an error of order $\alpha$), we obtain
\begin{equation}
T_\Sigma(\alpha)=2\pi\left(\frac{\pi}{4\sqrt{2}\alpha}\right)+o(\alpha)~=~\frac{\pi^2}{2\sqrt{2}\alpha}+o(\alpha).
\end{equation}
We also have $2\sqrt{E^2+2M^2}=2E+o(\alpha)$. Therefore, $T_S(M)$ is approximately given by
\begin{equation}\label{eq:ts}
T_S(M)\approx\frac{\pi^2}{4\sqrt{2}\alpha E}\approx \frac{\pi^2}{8M}.
\end{equation}
Eqs. \eqref{eq:tc} and \eqref{eq:ts} imply that
\begin{equation}
\frac{T_S(M)}{T_C(M)}\approx\frac{\pi}{4}\approx 0.78.
\end{equation}
In other words, there is an improvement of $22\%$ when using suboptimal controls for problem $(S)$ compared to the optimal ones for problem $(C)$. %We also observe that the sequence $(S_1)$ have the same frequency. See also Fig. \ref{fig:spiral}. 
 \begin{figure}[!ht]
\centering
\includegraphics[height=40mm]{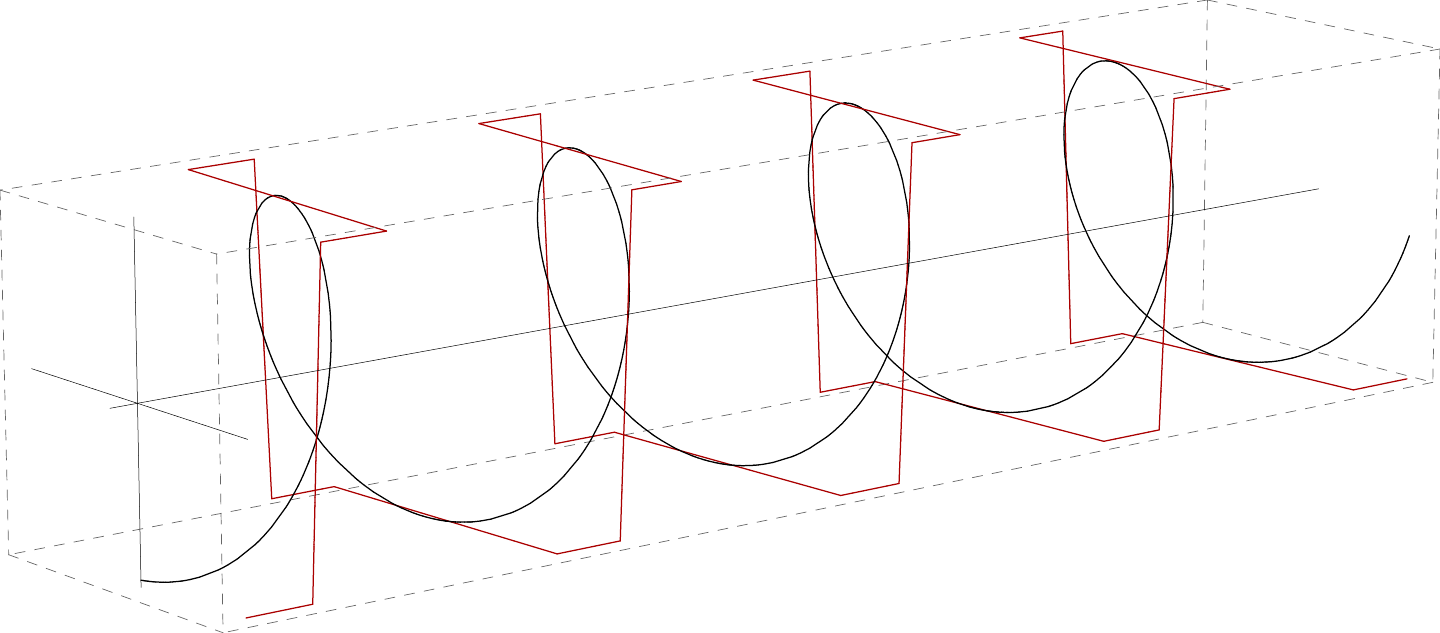}
\caption{Time evolution of $(u_1,u_2)$ in the limit $\alpha \rightarrow 0$ for the control system on the square (red) and the circle (black).}
\label{fig:spiral}
\end{figure}

\appendix
%\appendix
\section{Proof of Proposition \ref{bb}}\label{pf:bb}

Recall that the normalization for the co-vector $\lambda$ is given by $\lambda(0)=(\cos\theta, \sin\theta,0)$, and the corresponding initial conditions for the switching functions are given by Eqs. \eqref{eq:phi1} and \eqref{eq:phi2}. 
It follows from (ii) of Lemma \ref{eq:swi} that $\min_\theta \lambda_0 = -\sin\alpha$. 

\begin{proof}[Proof of \emph{(i)}]
By (ii) and (iii) of Lemma \ref{eq:swi} we have 
\begin{equation}\label{eq:ham}
	\phi_0(t) + \vert\phi_1(t)\vert + \vert \phi_2 (t) \vert + \lambda_0 = 0
\end{equation}
\begin{equation}\label{eq:ell}
	\frac{\phi_0^2(t)}{\cos^2\alpha} + \frac{\phi_1^2(t)}{\sin^2\alpha\sin^2\beta} + \frac{\phi_2^2(t)}{\sin^2\alpha\cos^2\beta} = 1
\end{equation}

We call the sets of $(\phi_0,\phi_1,\phi_2)$ that satisfy equations \eqref{eq:ham} and \eqref{eq:ell}, $S_\mathcal{H}$ and $S_{ad}$ respectively. By (i) of Lemma \ref{eq:swi} the solution of the adjoint system is  defined and unique in $[s,s+t]$. This solution must always lie in the intersection $S_\mathcal{H} \cap S_{ad}$, which is visualized in Fig. \ref{fig:intersection}. The surface defined by $S_\mathcal{H}$ is a union of four quarter-planes and may as such be spanned by rays starting at their common intersection,
\begin{equation}\label{eq:tip}
	\phi_0 = -\lambda_0, \quad \phi_1 = 0, \quad \phi_2 = 0.
\end{equation}

\begin{figure}[!ht]
 	\centering
	\includegraphics[height=60mm]{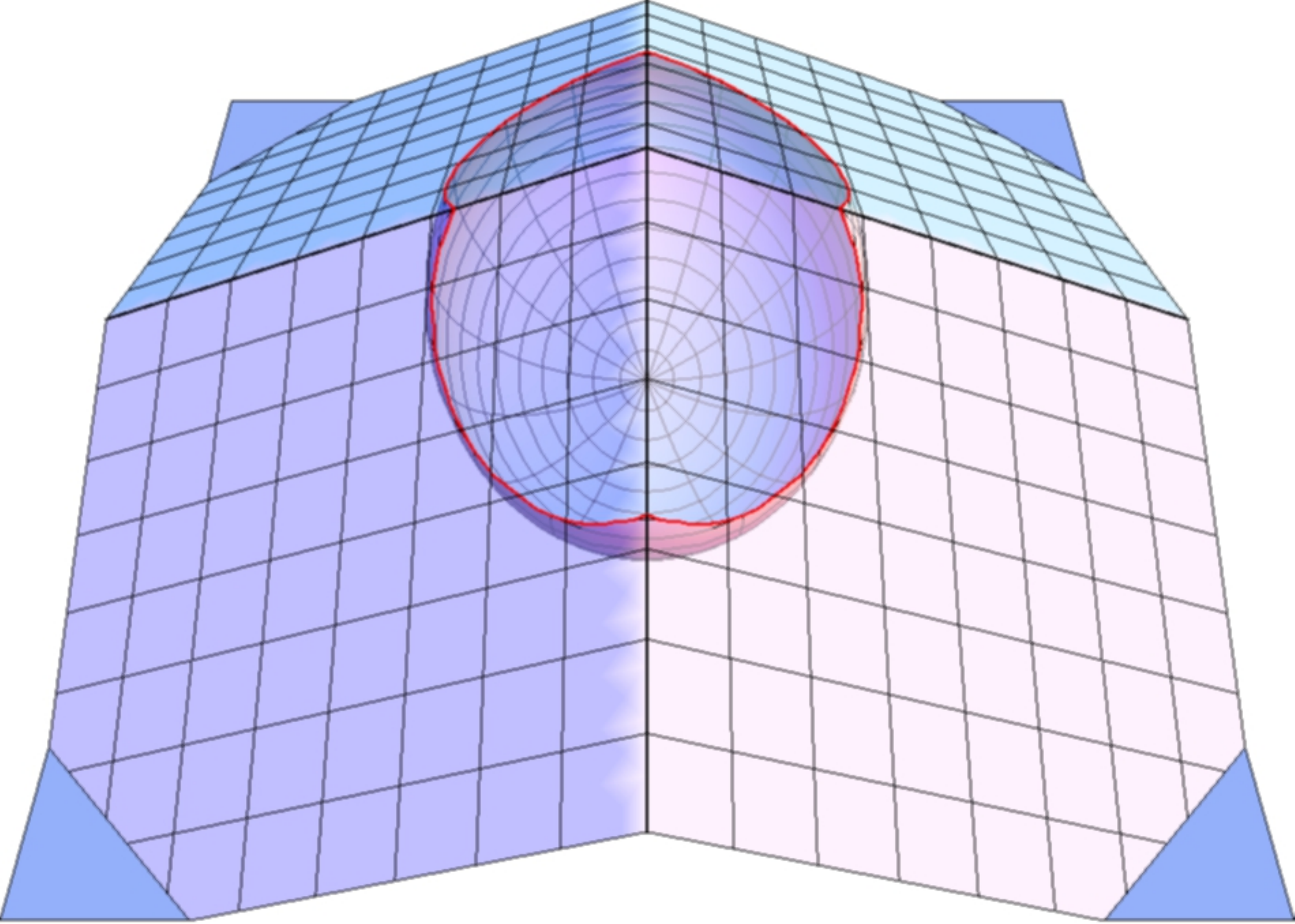}
	\caption{Intersection of the surfaces $S_{\mathcal{H}}$ and $S_{ad}$.}
	\label{fig:intersection}
\end{figure}

Assume by contradiction that $\phi_2(s) = \phi_2(s+t) = 0$. This implies that there exists a ray in $S_\mathcal{H}$ starting from \eqref{eq:tip} that intersects $S_{ad}$ more than once, since the solution must return to the $(\phi_0,\phi_1)$-plane without crossing itself. However, if $\alpha < \pi/4$, \eqref{eq:tip} belongs to the interior of $S_{ad}$, which is a strictly convex set. Thus any ray starting from \eqref{eq:tip} can only intersect $S_{ad}$ once.  Thus $\phi_2$ can never switch two times consecutively. The same reasoning holds for $\phi_1$.
\end{proof}
\smallskip
\begin{proof}[Proof of \emph{(ii)}]
Assume that the initial control is $(1,1)$, i.e. $\theta \in [\pi,3\pi/2]$. By (i) of Lemma \ref{eq:swi} with the initial condition given by $\lambda(0)$, for a given $\theta \in [\pi,3\pi/2]$, $s$ is the first switching time if and only if $f(\theta,s)=0$ with 
\begin{eqnarray*}
f(\theta, s)&:=&
\left(\frac{2\left(1-\sin^2\alpha \cos^2\beta\right)}{\sin^2\alpha}\cos\theta - \sin2\beta\sin\theta \right)\cos s \\
&-& \frac{2\cos\alpha}{\sin2\alpha}\sin\theta\sin s + \left(1+ \cos2\beta\right)\cos\theta + \sin2\beta\sin\theta.
\end{eqnarray*} 
It follows directly that $\displaystyle\cos s(\pi) = \frac{-\sin^2\alpha \cos^2\beta}{1 - \sin^2\alpha \cos^2\beta}$. Consider the following initial value problem
\begin{equation}\label{ini-s}
\left\{
\begin{array}{lcl}
\displaystyle s'&=&\displaystyle-\frac{\partial_{\theta}f(\theta,s)}{\partial_sf(\theta,s)},\smallskip\\
s(\pi)&=&\displaystyle\arccos\left(-\frac{\sin^2\alpha \cos^2\beta}{1 - \sin^2\alpha \cos^2\beta}\right),
\end{array}
\right.
\end{equation} where $s'$ denotes the derivative of $s$ with respect to $\theta$. It is easy to show that \eqref{ini-s} has a unique solution for $\theta \in [\pi,\pi+\epsilon[$ with $\epsilon >0$ small enough. 

The first step consists of showing \emph{global} existence of the solution to Eq. \eqref{ini-s} over the interval $[\pi,3/2\pi]$. We have
\begin{eqnarray*}
	\partial_\theta f_p(\theta,s) & = & -\left(\dfrac{2\left(1-\sin^2\alpha \cos^2\beta\right)}{\sin^2\alpha}\sin
	\theta + \sin2\beta\cos\theta \right)\cos s -\dfrac{2\cos\alpha}{\sin^2\alpha}\cos\theta\sin s \\
	&&- \left(1+ \cos2\beta\right)\sin\theta + \sin2\beta	\cos\theta, \\
	 \partial_s f_p(\theta,s) & = &  -\left(\dfrac{2\left(1-\sin^2\alpha \cos^2\beta\right)}{\sin^2\alpha}\cos
	 \theta - \sin2\beta\sin\theta \right)\sin s - \dfrac{2\cos\alpha}{\sin^2\alpha}\sin\theta\cos s. 
\end{eqnarray*} 
It is clear that there exists a constant $K > 0$ such that $\vert \partial_\theta f(\theta,s) \vert \leq K$. We need a uniform estimate for $1/\partial_s f(\theta,s)$. Let
\begin{equation*}
	a(\theta):= \frac{2\left(1-\sin^2\alpha \cos^2\beta\right)}{\sin^2\alpha}\cos\theta - \sin2\beta\sin\theta, 
	\quad b(\theta) := \frac{2\cos\alpha}{\sin^2\alpha}\sin\theta,
\end{equation*}
\begin{equation*}
c(\theta) := \left(1+ \cos2\beta\right)\cos\theta + \sin2\beta\sin\theta.
\end{equation*}
We have that $\displaystyle f(\theta,s) = a(\theta)\cos s - b(\theta) \sin s + c(\theta)$. It follows that 
\begin{equation*}\label{transformation}
	\begin{array}{lcl}
		\dfrac{f(\theta,s)}{\sqrt{a^2(\theta)+b^2(\theta)}} & = & \cos\gamma\cos s - \sin
		\gamma\sin s + \dfrac{c(\theta)}{\sqrt{a^2(\theta)+b^2(\theta)}} \\
		& = &  \cos(s + \gamma) +  \dfrac{c(\theta)}{\sqrt{a^2(\theta)+b^2(\theta)}},
	\end{array}
\end{equation*}
where $\displaystyle \cos\gamma := \frac{a(\theta)}{\sqrt{a^2(\theta)+b^2(\theta)}}$ and $\displaystyle\sin\gamma := \frac{b(\theta)}{\sqrt{a^2(\theta)+b^2(\theta)}}$.

Therefore, $\displaystyle \left| \partial_s f_p(\theta,s)\right| = \sqrt{a^2(\theta)+b^2(\theta)}\left| \sin(s + \gamma) \right|$.

Since $f(\theta,s) = 0$, we also have that 
\begin{equation*}
	\left| \sin(s + \gamma) \right| = \sqrt{1 - \cos^2(s + \gamma)} = \sqrt{1 -  \frac{c^2(\theta)}{a^2(\theta)+b^2(\theta)}},
\end{equation*}
which implies 
\begin{equation*}
	\left| \partial_s f_p(\theta,s)\right| = \sqrt{a^2(\theta)+b^2(\theta) - c^2(\theta)}.
\end{equation*}
We simplify the expression within the square root, and obtain
\begin{equation*}
	a^2(\theta) + b^2(\theta) - c^2(\theta) = \frac{2}{\sin^2\alpha}\left(2\cot^2\alpha - 2\cos{2\beta}
	\cos^2\theta - \sin2\beta\sin 2\theta\right).
\end{equation*}
We deduce that 
\begin{equation}\label{est-s}
	\left| \partial_s f_p(\theta,s)\right| \geq \frac{2\sqrt{\cot^2\alpha - \cos^2\beta}}{\sin\alpha},
\end{equation}
and it follows that 
\begin{equation}
	\left| \frac{ \partial_\theta f_p(\theta,s)}{ \partial_s f_p(\theta,s)} \right| \leq \frac{K \sin\alpha}{2\sqrt{\cot^2\alpha - \cos^2\beta}}.
\end{equation}
which implies that the solution of \eqref{ini-s} is globally defined on $[\pi,3\pi/2]$. 

Let $s(\cdot)$ be the solution to Eq. \eqref{ini-s}. The second step consists of showing that $s$ is decreasing on $[\pi,3\pi/2]$. We deduce from \eqref{est-s} that $\partial_s f(\theta,s)$ does not change sign on $[\pi,3\pi/2]$ and it is easy to show that it is positive. We now show that $\partial_\theta f(\theta,s)$ does not change sign neither. By contradiction, assume that $\partial_\theta f(\theta,s) = 0$ for some $s$. Together with the fact that $f(\theta,s)=0$, we obtain 

\begin{equation}
	\begin{pmatrix}
		c_1\cos\theta - c_4\sin\theta &  -c_2\sin\theta \\
		-c_1\sin\theta -c_4\cos\theta & -c_2\cos\theta 
	\end{pmatrix}
	\begin{pmatrix}
		\cos s \\
		\sin s		
	\end{pmatrix}
	=
	\begin{pmatrix}
		- c_3 \cos\theta - c_4\sin\theta \\
		c_3\sin\theta - c_4\cos\theta
	\end{pmatrix},
\end{equation}
with
\begin{equation*}
	\begin{array}{lclclcl}
		c_1 &:=& \dfrac{2\left(1-\sin^2\alpha \cos^2\beta\right)}{\sin^2\alpha}, && 
		c_3 &:=& 1+\cos2\beta, \\
		c_2 &:=& \dfrac{2\cos\alpha}{\sin^2\alpha}, && 
		c_4 &:=& \sin2\beta,
	\end{array}
\end{equation*}
which implies
\begin{eqnarray*}
	\cos s &=& -\dfrac{\sin^2\alpha\cos^2\beta}{1-\sin^2\alpha\cos^2\beta}, \\
	\sin s &=& \dfrac{\sin^2\alpha\sin\beta\cos\beta}{\left(1-\sin^2\alpha\cos^2\beta\right)\cos\alpha}.
\end{eqnarray*}
Therefore,
\begin{equation}\label{eq:trigone}
	\cos^2s+\sin^2s = \frac{\sin^2\alpha\cos^2\beta}{\left(1 - \sin^2\alpha\cos^2\beta \right)\cos^2\alpha}.
\end{equation}
However, it is easy to show that \eqref{eq:trigone} is less than one if $\alpha < \pi/4$, thus yielding a contradiction. Therefore $\partial_\theta f(\theta,s)$ does not vanish and $s'$ has a constant sign on $[\pi,3\pi/2]$. Using 
\begin{equation*}
	\cos s(\pi) = \frac{-\sin^2\alpha \cos^2\beta}{1 - \sin^2\alpha \cos^2\beta}, \qquad 
	\sin s(\pi) = \frac{\sqrt{1-2\sin^2\alpha \cos^2\beta}}{1 - \sin^2\alpha \cos^2\beta},
\end{equation*}
it is easy to show that $\partial_\theta f(\pi,s(\pi)) > 0$. Therefore $s'(\pi) < 0$, which implies that $s' < 0$ on $[\pi,3\pi/2]$. We conclude that  $\displaystyle \max_{\theta \in [\pi,3\pi/2]} s(\theta) = s(\pi)$.
The other cases with initial control equal to $(-1,-1)$, $(-1,1)$, or $(1,-1)$ are similar. 
\end{proof}\medskip

%%%%%%%%%%%%%%%%%%%%%%%%%%%%%%%%%%%%%%%%%%%

\begin{proof}[Proof of \emph{(iii)}]
Assume for instance $\phi_2(s)=\phi_1(s+t_1)=\phi_2(s+t_1+t_2)=0$. Let $X_{u_1u_2}:=F+u_1G_1+u_2G_2$. Then, we have
\begin{eqnarray}
&&\lambda(s+t_1)\exp(t_1X_{u_1u_2})G_2\exp(-t_1X_{u_1u_2})z(s+t_1)\nonumber\\
&=&\lambda(s+t_1)G_1z(s+t_1)\nonumber\\
&=&\lambda(s+t_1)\exp(-t_2X_{-u_1u_2})G_2\exp(t_2X_{-u_1u_2})z(s+t_1)\nonumber\\
&=&0.\label{eq:t1t2}
\end{eqnarray} Recall that the Lie algebra $(\soo(3),[,])$ is isomorphic to the Lie algebra $(\R^3,\wedge)$, where $\wedge$ denotes the vector product in $\R^3$, and we use the following isomorphism:
$$
i:\begin{pmatrix}
0&-c&b\\
c&0&-a\\
-b&a&0
\end{pmatrix}\rightarrow\begin{pmatrix}
a\\
b\\
c
\end{pmatrix}.$$ Then, \eqref{eq:t1t2} is equivalent to
\begin{equation*}
\det (i(G_1), e^{t_1X_{u_1u_2}}i(G_2),  e^{-t_2X_{-u_1u_2}}i(G_2)) = 0,\\
\end{equation*}
which may be simplified to the following equation
\begin{equation}\label{eq:det}
	\sin\left(\dfrac{t_2-t_1}{2}\right)D(t_1,t_2) = 0,
\end{equation}
where
\begin{equation*}
	\begin{array}{lcl}
		D(t_1,t_2) &:=& 2\sin \left(\frac{t_2}{2}\right) \left(4 \cos\alpha\cos\beta\cos\left(\frac{t_1}{2}\right)
		+\sin\left(\frac{t_1}{2}\right) \left((\cos2\alpha+3) \sin\beta-2\sin^2\alpha \sin3\beta\right)\right) \\
		&+&
		8\cos\left(\frac{t_2}{2}\right)\left(\cos\alpha\cos\beta \sin\left(\frac{t_1}{2}\right)+\sin\beta
		\cos \left(\frac{t_1}{2}\right)\right).
	\end{array}
\end{equation*}
It is easy to check that $D(t_1,t_2)\neq0$ if $(t_1,t_2) \in [0,\pi]^2$, which implies that $t_1 = t_2$. Lemma \ref{prop:assumption} guarantees that $t_1$ and $t_2$ indeed belong to $[0,\pi]$. Therefore, the duration between any two switchings is the same. 
\end{proof}

\begin{lemma}\label{prop:assumption}
Let $s$, $s+t_1$ and $s+t_1+t_2$ be three consecutive switching times. Then $(t_1,t_2) \in [0,\pi]^2$.
\end{lemma}
\begin{proof}[Proof of Lemma \ref{prop:assumption}]
To simplify the  notation, set $a_1:=\cos\alpha\cot\beta$, $a_2:=\cos\alpha\tan\beta$, and $a_3:=\cos\alpha\tan^2\alpha\sin\beta\cos\beta$. We have
$$P(u_1,u_2)=\begin{pmatrix}0& u_2a_1&-u_1a_2\\
-u_2a_3&0&a_2\\
u_1a_3&-a_1&0
\end{pmatrix},\quad \textrm{and }~a_1a_2+a_1a_3+a_2a_3=1.$$
Let $$v_1(u_1,u_2):=\begin{pmatrix}
\displaystyle\frac{a_2}{u_2 a_3}\smallskip\\
\displaystyle\frac{u_1a_2}{u_2 a_1}\smallskip\\
1
\end{pmatrix},
\quad v_2(u_1,u_2):=\begin{pmatrix}
\displaystyle-\frac{u_2a_1}{a_1+a_3}\smallskip\\
\displaystyle-\frac{u_1u_2a_3}{a_1+a_3}\\
1
\end{pmatrix},\quad
v_3(u_1,u_2):=\begin{pmatrix}
\displaystyle -\frac{u_1}{a_1+a_3}\smallskip\\
\displaystyle \frac{1}{a_1+a_3}\smallskip\\
0
\end{pmatrix},
$$It is straightforward that $Pv_1=0$, $Pv_2=v_3$, and $Pv_3=-v_2$.

Let $\displaystyle Q(v_1,v_2):=[v_1(u_1,u_2), v_2(u_1,u_2), v_3(u_1,u_2)]$. We have
$\displaystyle
Q^{-1}PQ=\begin{pmatrix}
0&0&0\\
0&0&-1\\
0&1&0
\end{pmatrix},
$ with
$\displaystyle Q^{-1}=
\begin{pmatrix}
 u_2a_1a_3&u_1u_2a_1a_3&a_1a_3\\
-u_2a_1a_3&-u_1u_2a_1a_3&a_2(a_1+a_3)\\
-u_1a_3&a_1&0
\end{pmatrix}.
$ Therefore,
\begin{equation}\label{eq:swi2}
\begin{pmatrix}
\displaystyle\phi_0(s+t)\\
\displaystyle\phi_1(s+t)\\
\displaystyle\phi_2(s+t)
\end{pmatrix}=Q(u_1,u_2)\begin{pmatrix}
1&0&0\\
0&\cos t& -\sin t\\
0&\sin t&\cos t
\end{pmatrix}Q^{-1}(u_1,u_2)\begin{pmatrix}
\displaystyle\phi_0(s)\\
\displaystyle\phi_1(s)\\
\displaystyle\phi_2(s)
\end{pmatrix},
\end{equation} where $(u_1,u_2)$ are the controls used in the interval $[s,s+t]$. 

Let $s$ and $s+t$ be two consecutive switching times. We show that $t\in[0,\pi]$. Without loss of generality, we can assume that $\phi_2(s)=0$. Then, we have $\phi_1(s+t)=0$. We also know that $\phi_0(s)=-\lambda_0-u_1\phi_1(s)$. Therefore, \eqref{eq:swi2} implies 
\begin{eqnarray}
\phi_1(s+t)&=&(p\cos\alpha+\lambda_0\cos\alpha\sin^2\alpha\sin^2\beta)\cos t+u_1u_2\sin\beta\cos\beta\sin^2\alpha(p+\lambda_0)\sin t\nonumber\\
&&-\lambda_0\cos\alpha\sin^2\alpha\sin^2\beta~ =~0,\label{eq:assH}
\end{eqnarray}where $p:=u_1\phi_1(s)=\vert \phi_1(s)\vert$ is the positive solution of 
\begin{equation}\label{eq:assH2}
(1+\frac{1}{\tan^2\alpha\sin^2\beta})p^2+2\lambda_0p+\lambda^2_0-\cos^2\alpha=0.
\end{equation}Note that \eqref{eq:assH} has exactly one positive solution if $\alpha<\pi/4$. Therefore, we obtain
\begin{equation}\label{eq:p}
p=\frac{-\lambda_0\sin^2\alpha\sin^2\beta+\cos\alpha\sin\alpha\sin\beta\sqrt{\Delta}}{\sin^2\alpha\sin^2\beta+\cos^2\alpha},
\end{equation}with $\displaystyle \Delta:=1-\lambda_0^2-\sin^2\alpha\cos^2\beta$. Substituting \eqref{eq:p} in \eqref{eq:assH}, we obtain after simplification
\begin{equation}\label{eq:assH3}
c_1\cos t+c_2\sin t+c_3=0,
\end{equation} where
\begin{eqnarray*}
c_1&:=&\frac{\cos\alpha\sqrt{\Delta}}{\sin\alpha}-\lambda_0\sin\beta\sin^2\alpha\cos^2\beta,\\
c_2&:=&u_1u_2\cos\beta(\sin\alpha\sin\beta\sqrt{\Delta}+\lambda_0\cos\alpha),\\
c_3&:=&-\lambda_0\sin\beta(\sin^2\beta+\cos^2\alpha\cos^2\beta).
\end{eqnarray*}
Eq. \eqref{eq:assH3} is equivalent to
\begin{equation}\label{eq:assH4}
(c_3-c_1)\tan^2\frac{t}{2}+2c_2\tan\frac{t}{2}+(c_3+c_1)=0,
\end{equation} and we only need to show that \eqref{eq:assH4} has positive solution. Prove that $\displaystyle \frac{c_3+c_1}{c_3-c_1}<0$, which implies that \eqref{eq:assH4} has exactly one positive solution. As $\lambda_0<0$, it is clear that $c_3+c_1>0$. Therefore, we only need to check the sign of $c_3-c_1$.
\begin{eqnarray*}
c_3-c_1&=&-\frac{\lambda_0\sin\alpha\sin\beta(1-2\sin^2\alpha\cos^2\beta)+\cos\alpha\sqrt{\Delta}}{\sin\alpha}.
\end{eqnarray*} Let $\bar{\lambda}_0:=-\lambda_0>0$. We have
\begin{eqnarray}
&&\cos\alpha\sqrt{1-\bar{\lambda}^2_0-\sin^2\alpha\cos^2\beta}-\bar{\lambda}_0\sin\alpha\sin\beta(1-2\sin^2\alpha\cos^2\beta)\nonumber\\
&=&\frac{\cos^2\alpha(1-\bar{\lambda}^2_0-\sin^2\alpha\cos^2\beta)-\bar{\lambda}^2_0\sin^2\alpha\sin^2\beta(1-2\sin^2\alpha\cos^2\beta)^2}{\cos\alpha\sqrt{1-\bar{\lambda}^2_0-\sin^2\alpha\cos^2\beta}+\bar{\lambda}_0\sin\alpha\sin\beta(1-2\sin^2\alpha\cos^2\beta)}\nonumber\\
&\geq&\frac{\cos^2\alpha(1-\bar{\lambda}^2_0-\sin^2\alpha\cos^2\beta)-\bar{\lambda}^2_0\sin^2\alpha\sin^2\beta(1-2\sin^2\alpha\cos^2\beta)}{\cos\alpha\sqrt{1-\bar{\lambda}^2_0-\sin^2\alpha\cos^2\beta}+\bar{\lambda}_0\sin\alpha\sin\beta(1-2\sin^2\alpha\cos^2\beta)}\nonumber\\
&=&\frac{\cos^2\alpha(1-\sin^2\alpha\cos^2\beta)-\bar{\lambda}^2_0(\cos^2\alpha+\sin^2\alpha\sin^2\beta-2\sin^4\alpha\sin^2\beta\cos^2\beta)}{\cos\alpha\sqrt{1-\bar{\lambda}^2_0-\sin^2\alpha\cos^2\beta}+\bar{\lambda}_0\sin\alpha\sin\beta(1-2\sin^2\alpha\cos^2\beta)}\label{eq:assH6}.
\end{eqnarray}
Using the fact that,  for $\alpha<\pi/4$, $\cos^2\alpha>\bar{\lambda}^2_0$, and
\begin{eqnarray*}
&&(1-\sin^2\alpha\cos^2\beta)-(\cos^2\alpha+\sin^2\alpha\sin^2\beta-2\sin^4\alpha\sin^2\beta\cos^2\beta)\\
&=&2\sin^4\alpha\sin^2\beta\cos^2\beta~>~ 0,
\end{eqnarray*} we conclude that \eqref{eq:assH6} is positive, which implies that $c_3-c_1<0$. 
\end{proof}

\section{Expression of $v(s)$}\label{app:v}

We establish in this section the expressions for the duration between switchings on a normal extremal as a function of the the first switching time $s$.  For convenience we set  $c_\alpha$ and $s_\alpha$ to denote respectively $\cos\alpha$ and $\sin\alpha$.

\begin{lemma}
If the extremal trajectory is starting from the north pole with control $(1,1)$ or $(-1,-1)$ (resp.  with control  $(1,-1)$ or $(-1,1)$) then  $v(s)=v_1(s)$ (resp.  $v(s)=v_2(s)$), where
\beq
	v_i(s) = \arccos \left[ \frac{A_i(s) + B_i(s)\sqrt{C_i(s) + E_i}}{D_i(s) + F_i}  \right], \quad i = 1,2,
\eeq 
with
\begin{eqnarray*}
	A_1(s) &=& 4s^6_\alpha s^2_{2\beta}+8 s^2_{2 \alpha } s_{\beta } (c_{\alpha } c_{\beta } \sin (s)+s_{\beta } \cos (s))+2s^4_
	{\alpha} \left((3+c_{2 \alpha})s^2_{2\beta} \cos(2s) + 2c_{\alpha} s_{4 \beta } \sin(2 s)\right),\\
	A_2(s) &=& 4s^6_\alpha s^2_{2\beta}+8 s^2_{2 \alpha } c_{\beta } \left(c_{\alpha } s_{\beta } \sin (s)+c_{\beta } \cos (s)\right)+2s^4_
	{\alpha} \left((3+c_{2 \alpha})s^2_{2\beta} \cos(2s) - 2c_{\alpha} s_{4 \beta } \sin(2 s)\right),
	\\
	B_1(s) & = & \sqrt{2} s^2_\alpha c_\beta \left( c_\beta \sin (s) + c_\alpha s_\beta (\cos (s) - 1) \right),
	\\
	B_2(s) & = & \sqrt{2} s^2_\alpha s_\beta \left( s_\beta \sin (s) + c_\alpha c_\beta (\cos (s) - 1) \right),
	\\
	C_1(s) & = & 256 s^2_\alpha c_\alpha c_\beta \left(3 +c_{2 \alpha} + 2s^2_\alpha c_ {2 \beta}\right) 
	\left( 
	c_\alpha c_\beta \cos (s) - s_\beta \sin(s)
	\right) 
	\\
	&+& 64 s^4_\alpha s^2_\beta \left( 
	\left((3 + c_{2\alpha}) c_{2\beta} - 4s^2_\alpha \right) \cos(2s) 
	- 4  c_\alpha s_{2\beta} \sin (2 s)
	\right),
	\\
	C_2(s) & = & 256 s^2_\alpha c_\alpha s_\beta  \left(3 +c_{2 \alpha} - 2s^2_\alpha c_ {2 \beta}\right) 
	 \left(  c_\alpha s_\beta \cos(s) - c_\beta \sin(s)  \right) 
	\\
	&- & 64 s^4_\alpha c^2_\beta \left( \left(\left(3+c_{2 \alpha}\right) c_{2 \beta}+4s^2_\alpha\right)\cos (2 s)  + 4 c_\alpha s_
	{2\beta}\sin (2 s) \right),
	\\
	D_1(s) &=& 16s^2_\alpha c_\alpha c_\beta\left(3 +c_{2 \alpha}+ 2 s^2_\alpha c_{2 \beta}\right)\left(s_\beta\sin (s) - 
	c_\alpha c_\beta \cos (s) \right),
	\\
	D_2(s) &=& 16s^2_\alpha c_\alpha s_\beta \left(3 +c_{2 \alpha} - 2 s^2_\alpha c_{2 \beta}\right)\left(
	c_\beta \sin (s)- c_\alpha s_\beta\cos (s)	
	\right),
	\\
	E_1 &=& 234 + 384 s^4_\alpha c_{2\beta} -16 s ^4_\alpha  c _{4\beta}(1+ 3c_{2 \alpha} ) +205 c _{2\alpha}+70 c_{4 \alpha}+3c_
	{6\alpha},
	\\
	E_2 &=& 234 - 384 s^4_\alpha c_{2\beta} -16 s ^4_\alpha c _{4\beta}  (1+ 3c_{2 \alpha} ) +205 c _{2\alpha}+70 c_{4 \alpha}+3c_
	{6\alpha},
	\\
	F_1 &=& -17 -16 s^4_\alpha c_{2 \beta}+c_{2 \alpha} \left( 4s^4_\alpha c_{4 \beta} - \frac{39}{4} \right) 
	- 5c_{4 \alpha} - \frac{1}{4} c_{6 \alpha},
	\\
	F_2 &=& -17 +16 s^4_\alpha c_{2 \beta}+c_{2 \alpha} \left( 4s^4_\alpha c_{4 \beta} - \frac{39}{4} \right) 
	- 5c_{4 \alpha} - \frac{1}{4} c_{6 \alpha}.
\end{eqnarray*}
\end{lemma}

\begin{proof}
Similar to the proof of (iii) of Proposition \ref{bb}, assume
$$\phi_2(s)=\phi_1(s+v)=0,$$ where $s$ is the first switching time. Therefore, 
$$\lambda(s)G_2x(s)=\lambda(s)\exp(-vX_{+-})G_1\exp(vX_{+-})x(s)=0,$$ which implies that
\begin{eqnarray}
g(s,v)&:=&\det(i(G_2), i(\exp(-vX_{+-})G_1\exp(vX_{+-})), \exp(sX_{++})x(0))\nonumber\\
&=&0.
\end{eqnarray}
After simplification, we obtain
\begin{equation}
a(s)\cos v+b(s)\sin v+c(s)=0,
\end{equation}with
\begin{eqnarray*}
a(s)&:=&-\cos^2\beta\cos s+\cos\alpha\sin\beta\cos\beta \sin s-\cot^2\alpha,\\
b(s)&:=&-\cos\alpha\sin\beta\cos\beta \cos s-\cos^2\beta \sin s+\cos\alpha\sin\beta\cos\beta,\\
c(s)&:=&-\sin^2\beta \cos s-\cos\alpha\sin\beta\cos\beta \sin s.
\end{eqnarray*} The result follows.
\end{proof}

\begin{lemma}\label{lemma:period}
For the special case of $\beta = \pi/4$, we obtain a simpler expression:
	\beq\label{eq:period}
		v(s) = \arccos{\left[\frac{d-A(s)-B(s)-C(s)}{e-A(s)+B(s)}\right]},
	\eeq
	where
	\beqn
		\begin{array}{ll}
			A(s) = 8 c_\alpha s^2_\alpha \sin(s), & d = s^2_{2\alpha}, \\
			B(s) = 2 s^2_{2\alpha}\cos(s), & e = 5+2c_{2\alpha}+c_{4\alpha}, \\
			C(s) = 4 s^4_\alpha\cos(2s).
		\end{array}
	\eeqn
\end{lemma}

From Lemma \ref{lemma:period} and (ii) of Proposition \ref{bb}, we deduce the following corollary.

\begin{corollary}\label{coro:0smax}
For $\displaystyle \beta=\frac{\pi}{4}$, we have $\displaystyle v(0)=v(s_{\max})=s_{\max}.$
\end{corollary}
\begin{proof}
From (ii) of Proposition \ref{bb}, we have
\begin{equation}\label{eq:smax-sym}
\cos s_{\max}=-\frac{\sin^2\alpha}{1+\cos^2\alpha},\quad \sin s_{\max}=\frac{2\cos\alpha}{1+\cos^2\alpha}.
\end{equation} Substituting \eqref{eq:smax-sym} into \eqref{eq:period}, it is easy to check that $\cos (v(s_{\max}))=\cos (s_{\max})$. In a similar manner, we check that $v(0)=s_{\max}$.
\end{proof}

\section{Some computational lemmas for the case $\beta=\pi/4$}
\label{app:com}

The following result is a consequence of Corollary \ref{coro:0smax}. It can be checked by direct computation.

\begin{lemma}\label{lemma:Msmax0}
Let $\displaystyle \theta(\alpha):=\arcsin\left(\frac{-2\sqrt{2}\sin\alpha\cos\alpha}{1+\cos^2\alpha}\right)$. Then, we have
\begin{eqnarray*}
e^{v(s_{\max})X_{++}}&=&\begin{pmatrix}0&-1&0\\ \cos\theta(\alpha)&0&\sin\theta(\alpha)\\-\sin\theta(\alpha)&0&\cos\theta(\alpha)\end{pmatrix},\quad
e^{v(s_{\max})X_{+-}}~=~\begin{pmatrix}0&-\cos\theta(\alpha)&-\sin\theta(\alpha)\\ 1&0&0\\ 0&-\sin\theta(\alpha)&\cos\theta(\alpha)\end{pmatrix},\\
e^{v(s_{\max})X_{--}}&=&\begin{pmatrix}0&-1&0\\ \cos\theta(\alpha)&0&-\sin\theta(\alpha)\\\sin\theta(\alpha)&0&\cos\theta(\alpha)\end{pmatrix},\quad
e^{v(s_{\max})X_{-+}}~=~\begin{pmatrix}0&-\cos\theta(\alpha)&\sin\theta(\alpha)\\ 1&0&0\\ 0&\sin\theta(\alpha)&\cos\theta(\alpha)\end{pmatrix}.
\end{eqnarray*}
\end{lemma}

Note that $\displaystyle\sin\theta(\alpha)=\frac{-2\sqrt{2}\sin\alpha\cos\alpha}{1+\cos^2\alpha}$ and $\displaystyle\cos\theta(\alpha)=\frac{3\cos^2\alpha-1}{1+\cos^2\alpha}$. By Lemma \ref{lemma:Msmax0}, we obtain an exact expression for $\bar{M}(s_{\max},\alpha)$.
\begin{corollary}\label{lemma:Msmax}
We have
\begin{equation}\label{eq:Msmax}
\bar{M}(0,\alpha)=\bar{M}(s_{\max},\alpha)=\begin{pmatrix}
1&0&0\\
0&\cos4\theta(\alpha)&\sin4\theta(\alpha)\\
0&-\sin4\theta(\alpha)&\cos4\theta(\alpha)
\end{pmatrix},
\end{equation}where $\theta(\alpha)$ is defined in Lemma \ref{lemma:Msmax0}.
\end{corollary}

In other words, $\bar{M}(0,\alpha)$ and $\bar{M}(s_{\max},\alpha)$ are rotations around $x_1-$axis of angle $4\theta(\alpha)$. This fact is crucial for the derivation of suboptimal strategies presented in Sec. \ref{sec:suboptimal}. It is worth noticing that formula \eqref{eq:Msmax} is exact for any $\displaystyle\alpha$ smaller than ${\pi}/{4}$. If $\alpha$ is small enough, we have $4\theta(\alpha)=-4\sqrt{2}\alpha+o(\alpha)$ which agrees the first order approximation used in the proof of Proposition  \ref{diofa}.

%\begin{lemma}\label{lemma:Mrot}
%Let $\displaystyle \theta(\alpha):=\arcsin\left(\frac{-2\sqrt{2}\sin\alpha\cos\alpha}{1+\cos^2\alpha}\right)$. We have
%$$C(\alpha)=\cos (4\theta(\alpha)),\qquad S(\alpha)=-\sin (4\theta(\alpha)).$$
%\end{lemma}
%
%From Corollary \ref{lemma:Msmax} and Lemma \ref{lemma:Mrot}, we obtain immediately the following  result, 
%\begin{corollary}\label{coro:M}
%$\bar{M}(s_{\max},\alpha)$ is a rotation around $x_1-$axis of angle $4\theta(\alpha)$.
%\end{corollary}

\begin{corollary}\label{coro:spgc}
Starting from the north pole, the switching points of the extremals having their first switching at $s_{\max}$ are located on the great circles passing through $N$ and containing the $x_1-$ or $x_2-$axis.
\end{corollary}
\begin{proof}
Note that  the switching points of these extremals are given by $$m_2(\alpha)\bar{M}^n(s_{\max},\alpha)m_1(\alpha)N,$$where $n$ is an integer, $m_1(\alpha)$ denotes one of four exponentials in Lemma \ref{lemma:Msmax0}, and $m_2(\alpha):=e^{s_{\max}X_{++}}$, or $e^{s_{\max}X_{+-}}e^{s_{\max}X_{++}}$, or $e^{s_{\max}X_{--}}e^{s_{\max}X_{+-}}e^{s_{\max}X_{++}}$. Corollary \ref{coro:spgc} is then proved by induction on $n$, using Lemma \ref{lemma:Msmax0} and Corollary \ref{lemma:Msmax}.
\end{proof}
%\begin{lemma}\label{lemma:CN}
%Starting from the north pole with the control $(1,1)$, we have
%\begin{equation*}
%e^{s_{\max}X_{++}}N=\begin{pmatrix}
%0\\
%\displaystyle\frac{-2\sqrt{2}\sin\alpha\cos\alpha}{1+\cos^2\alpha}\smallskip\\
%\displaystyle\frac{3\cos^2\alpha-1}{1+\cos^2\alpha}
%\end{pmatrix}=\begin{pmatrix}
%0\\
%\sin(\theta(\alpha))\\
%\cos(\theta(\alpha))
%\end{pmatrix}.
%\end{equation*}
%\end{lemma}\smallskip
%

The following two lemmas are valid for $\alpha$ small enough and $\beta=\pi/4$.
\begin{lemma}\label{eq:vs-small-alpha}
Let $v(s)$ be the second switching time as a function of the first one $s$. For $\alpha$ small enough, we have
\begin{equation}
v(s)=\frac{\pi}{2}+f_1(s)~\alpha^2+f_2(s)\alpha^4+O(\alpha^6),\quad \textrm{for }s\in[0,s_{\max}],
\end{equation}where $\displaystyle s_{\max}=\frac{\pi}{2}+\frac{1}{2}\alpha^2+\frac{1}{12}\alpha^4+O(\alpha^5)$, and 
\begin{eqnarray*}
f_1(s)&:=&-\frac{1}{2}+\cos s+\sin s,\\
f_2(s)&:=&\frac{25}{24}-\frac{1}{3}\sin s+\frac{1}{6}\cos s+\cos s\sin s-\cos^2 s.
\end{eqnarray*}
\end{lemma}

\begin{lemma}\label{eq:prod}
We have
\begin{eqnarray*}
\bar{M}(s,\alpha)&=&\begin{pmatrix}
1+f_3(s)\alpha^4,&f_4(s)\alpha^2+f_5(s)\alpha^4&f_6(s)\alpha^3\\
-f_4(s)\alpha^2-f_5(s)\alpha^4,&1-16\alpha^2+f_7(s)\alpha^4&-4\sqrt{2}\alpha-f_8(s)\alpha^4\\
f_6(s)\alpha^3,&4\sqrt{2}\alpha+f_8(s)\alpha^4,&1-16\alpha^2+f_9(s)\alpha^4
\end{pmatrix}+O(\alpha^5)\nonumber\smallskip\\
&:=&\bar{M}_a(s,\alpha)+O(\alpha^5),
\end{eqnarray*}where 
\begin{eqnarray*}
f_3(s)&:=&16\sin s-16\cos s\sin s-16+16\cos s,\quad f_4(s)~:=~4-4\cos s-4\sin s,\\
f_5(s)&:=&-\frac{70}{3}+\frac{58}{3}\cos s+\frac{64}{3}\sin s+4\cos^2s,\quad f_6(s)~:=~8\sqrt{2}(-1+\cos s+\sin s),\\
f_7(s)&:=&\frac{112}{3}-16\cos s~\sin s,\quad f_8(s)~:=~\frac{2\sqrt{2}}{3}(-34+3\cos s+3\sin s),\\
f_9(s)&:=&\frac{160}{3}-16\sin s-16\cos s.
\end{eqnarray*}Moreover, 
%\begin{equation}
$\displaystyle \bar{M}^{-1}(s,\alpha)=\bar{M}_a^{-1}(s,\alpha)+O(\alpha^5)=\bar{M}_a^T(s,\alpha)+O(\alpha^5).$
%\end{equation}
\end{lemma}\smallskip
%\begin{corollary}
%As a consequence of Lemma \ref{eq:vs-small-alpha} and Lemma \ref{eq:prod}, $M(0,\alpha)$ and $M(s_{\max},\alpha)$ are equal to 
%\begin{eqnarray}
%\bar{M}(\alpha)&:=&\begin{pmatrix}
%1,&0,&0\\
%0,&\displaystyle 1-16~\alpha^2+\frac{112}{3}~\alpha^4,&\displaystyle-4\sqrt{2}~\alpha+\frac{62}{3}\sqrt{2}~\alpha^3\\
%0,&\displaystyle 4\sqrt{2}~\alpha-\frac{62}{3}\sqrt{2}~\alpha^3,&\displaystyle1-16~\alpha^2+\frac{112}{3}\alpha^4
%\end{pmatrix}+O(\alpha^5),\nonumber\\
%&=&\begin{pmatrix}
%1,&0,&0\\
%0,&\cos \theta,&-\sin \theta\\
%0,&\sin\theta,&\cos\theta
%\end{pmatrix}+O(\alpha^5),\label{eq:Mbar}
%\end{eqnarray}with
%\begin{equation}\label{eq:theta}
%\theta:=4\sqrt{2}\alpha+\frac{2}{3}\sqrt{2}\alpha^3+O(\alpha^5).
%\end{equation}
%\end{corollary}


\begin{thebibliography}{0}%
\makeatletter
\providecommand \@ifxundefined [1]{%
 \@ifx{#1\undefined}
}%
\providecommand \@ifnum [1]{%
 \ifnum #1\expandafter \@firstoftwo
 \else \expandafter \@secondoftwo
 \fi
}%
\providecommand \@ifx [1]{%
 \ifx #1\expandafter \@firstoftwo
 \else \expandafter \@secondoftwo
 \fi
}%
\providecommand \natexlab [1]{#1}%
\providecommand \enquote  [1]{``#1''}%
\providecommand \bibnamefont  [1]{#1}%
\providecommand \bibfnamefont [1]{#1}%
\providecommand \citenamefont [1]{#1}%
\providecommand \href@noop [0]{\@secondoftwo}%
\providecommand \href [0]{\begingroup \@sanitize@url \@href}%
\providecommand \@href[1]{\@@startlink{#1}\@@href}%
\providecommand \@@href[1]{\endgroup#1\@@endlink}%
\providecommand \@sanitize@url [0]{\catcode `\\12\catcode `\$12\catcode
  `\&12\catcode `\#12\catcode `\^12\catcode `\_12\catcode `\%12\relax}%
\providecommand \@@startlink[1]{}%
\providecommand \@@endlink[0]{}%
\providecommand \url  [0]{\begingroup\@sanitize@url \@url }%
\providecommand \@url [1]{\endgroup\@href {#1}{\urlprefix }}%
\providecommand \urlprefix  [0]{URL }%
\providecommand \Eprint [0]{\href }%
\providecommand \doibase [0]{http://dx.doi.org/}%
\providecommand \selectlanguage [0]{\@gobble}%
\providecommand \bibinfo  [0]{\@secondoftwo}%
\providecommand \bibfield  [0]{\@secondoftwo}%
\providecommand \translation [1]{[#1]}%
\providecommand \BibitemOpen [0]{}%
\providecommand \bibitemStop [0]{}%
\providecommand \bibitemNoStop [0]{.\EOS\space}%
\providecommand \EOS [0]{\spacefactor3000\relax}%
\providecommand \BibitemShut  [1]{\csname bibitem#1\endcsname}%
\let\auto@bib@innerbib\@empty
%</preamble>
\end{thebibliography}%


\begin{thebibliography}{99}
\bibitem{agra-book} A. Agrachev, Y. Sachkov,  {\it Control Theory from the Geometric Viewpoint},
Encyclopedia of Mathematical Sciences, v.87, Springer (2004).
\bibitem{allen} L. Allen, J. H. Eberly,{\it Optical Resonance and Two-Level Atoms}, Wiley, New York (1975).
\bibitem{S1} B. Bonnard, O. Cots, N. Shcherbakova and D. Sugny, ``The energy minimization problem for two-level dissipative quantum systems," J. Math. Phys. 51, 092705 (2010).
\bibitem{BS}B. Bonnard, D. Sugny, \emph{Optimal Control with Applications in Space and Quantum Dynamics}, AIMS (2012).
\bibitem{BCC} U. Boscain, T. Chambrion, and G. Charlot, ``Nonisotropic 3-level Quantum Systems: Complete Solutions for Minimum Time and Minimal Energy," Discrete Contin. Dyn. Syst. Ser. B 5, 957--990 (2005). 
\bibitem{q2} U. Boscain, G. Charlot, J.-P. Gauthier, S. Gu\'erin, and H.-R. Jauslin, ``Optimal control in laser-induced population transfer for two- and three- level quantum systems," J. Math. Phys. 43, 2107 (2002).
\bibitem{B-M} U. Boscain, P. Mason, ``Time Minimal Trajectories for a Spin 1/2 Particle in a Magnetic Field," J. Math. Phys. 47, 062101 (2006). 
\bibitem{libro} U. Boscain, B, Piccoli, \emph{Optimal Synthesis  for Control Systems on $2-D$ Manifolds}, Springer, SMAI, v.43 (2004).
\bibitem{bressan}A. Bressan, ``The generic local time optimal stabilizing control in dimension $3$," SIAM J. Control Optim. 24, 77--190 (1986).
\bibitem{choen} C. Cohen-Tannoudji, B. Diu, F. Laloe, {\it Quantum Mechanics}, Hermann,  New York (1977).
\bibitem{daless} D. D'Alessandro and M. Dahleh, ``Optimal control of two-level quantum systems," IEEE Trans. on Automatic Control 46, 866-876 (2001).
\bibitem{KBG}N. Khaneja, R. Brockett, and S. J. Glaser, ``Time optimal control in spin systems," Phys. Rev. A. 63, 032308 (2001).
\bibitem{lap}M. Lapert, Y. Zhang, S. J. Glaser, D. Sugny, ``Towards the Time-Optimal Control of Dissipative Spin $1/2$ Particles in Nuclear Magnetic Resonance," J. Phys. B 44, 154014 (2011).
\bibitem{lev}M. Levitt, \emph{Spin Dynamics} $2^{\textrm{nd}}$ed. Wiley (2008).
\bibitem{limit} P. Mason, R. Salmoni,  Boscain, Y. Chitour, ``Limit Time Optimal Syntheses for a control-affine system on $S^2$," SIAM J. Control Optim. 47, 111--143 (2008).
\bibitem{piccoli-sussman}  B. Piccoli and H. J. Sussmann, ``Regular synthesis and sufficiency conditions for optimality," SIAM, J. Contr. Optim. 39, 359--410 (2000).
\bibitem{pont-book} L. S. Pontryagin, V. G. Boltyanskii, R. V. Gamkrelidze, and E. F. Mishchenko, {\it The Mathematical Theory of Optimal Processes}, v. 4 in L. S. Pontryagin Selected Works, Gordon and Breach Science Publishers, (1986).  
\bibitem{sus1}H. Sussmann, ``The structure of time-optimal trajectories for single-input systems in the plane: the $C^\infty$ nonsingular case," SIAM J. Control Optim. 25, 433--465 (1987).
\bibitem{sus2}H. Sussmann, ``Regular synthesis for time-optimal control of single-input analytic systems in the plane," SIAM J. Control Optim. 25, 1145--1162 (1987).


\end{thebibliography}
\end{document}